\documentclass[a4paper,11pt]{article}
\usepackage[pagewise]{lineno}
\usepackage{amsmath}
\usepackage{amsmath}
\usepackage{amssymb}
\usepackage{mathrsfs}
\usepackage{amsfonts}
\usepackage{amsthm}
\usepackage{multirow}
\usepackage{float}
\usepackage{subfigure}
\usepackage[]{caption2}

\renewcommand{\thesubfigure}{Figure \arabic{figure}\alph{subfigure}}
\makeatletter \renewcommand{\@thesubfigure}{\thesubfigure \space}
\renewcommand{\p@subfigure}{} \makeatother

\usepackage{pstricks, pst-node, pst-text, pst-3d,psfrag}
\usepackage{graphicx}
\usepackage{indentfirst}
\usepackage{enumerate}
\usepackage{cite}
\usepackage[colorlinks=true]{hyperref}
\hypersetup{urlcolor=blue, citecolor=red}
\usepackage{soul}
\usepackage{ulem}
\usepackage{cancel}
\usepackage{footmisc}
\usepackage{slashed}
\usepackage{extarrows}
\usepackage{xcolor}

\theoremstyle{plain}
\newtheorem{thm}{Theorem}[section]
\newtheorem{prop}[thm]{Proposition}
\newtheorem{cor}[thm]{Corollary}
\newtheorem{lem}[thm]{Lemma}
\theoremstyle{definition}
\newtheorem{defn}[thm]{Definition}

\theoremstyle{remark}

\newtheorem*{hypo*}{\textbf{Hypothesis}}

\newtheorem*{claim*}{Claim}
\newcommand{\tabincell}[2]{\begin{tabular}{@{}#1@{}}#2\end{tabular}}
\numberwithin{equation}{section}

\makeatletter
\@addtoreset{equation}{section}
\makeatother

\def\@biblabel#1{#1}
\makeatother
.5 \textwidth 155mm \textheight 230mm \topmargin=-1cm
\newenvironment{demo*}{\vspace{3mm}\noindent{\bf Proof.}}{\hfill $\Box$ \vspace{3mm}}

\begin{document}
\title{Stochastic Stability of Monotone Dynamical Systems. I. The Irreducible Cooperative Systems}

\setlength{\baselineskip}{16pt}

\author{
Jifa Jiang\thanks{Supported by the National Natural Science Foundation of China (No. 12571171 and 12171321).} 
\\[2mm]
College of Mathematics and Statistics\\
Henan Normal University\\
Xinxiang, Henan, 453007, P. R. China
\\[2mm]
Xi Sheng and Yi Wang\thanks{Supported by the National Key R\&D Program of China (No.2024YFA1013603, 2024YFA1013600), the National Natural Science Foundation of China (No.12331006), and the Strategic Priority Research Program of CAS (No.XDB0900100).} 
\\[2mm]
School of Mathematical Sciences\\
University of Science and Technology of China\\
Hefei, Anhui, 230026, P. R. China
}

\date{}
\maketitle

\begin{abstract}
The current series of papers is concerned with stochastic stability of monotone dynamical systems by identifying the basic dynamical units that can survive in the presence of noise interference.
In the first of the series, for the cooperative and irreducible systems, we will establish the stochastic stability of a dynamical order, that is, the zero-noise limit of stochastic perturbations will be concentrated on a simply ordered set consisting of Lyapunov stable equilibria.
In particular, we utilize the Freidlin--Wentzell large deviation theory to gauge the rare probability in the vicinity of unordered chain-transitive invariant set on a nonmonotone manifold.
We further apply our theoretic results to the stochastic stability of classical positive feedback systems by showing that the zero-noise limit is a convex combination of the Dirac measures on a finite number of asymptotically stable equilibria although such system may possess nontrivial periodic orbits.
\end{abstract}

\section{Introduction}
Monotone dynamical systems, stemming from the groundbreaking work of M. W. Hirsch, are characterized by the presence of a comparison principle that is aligned with a closed partial order relation (induced by a convex cone) within the state space.
Such systems, due to Smale \cite{S76}, may harbor orbits that exhibit arbitrary complicated behavior, including horseshoes (see also in e.g., \cite{Dancer,Smith1,Smith2,W-D-10,W-J-01,W-Y-22} and references therein). On the other hand,
the signature results of monotone dynamical systems indicate that the forward orbit of almost every (i.e., generic/prevalent) initial
state converges in both the topological sense (\!\!\cite{Polacik89,P-T-91,Hirsch,Hirsch2,H-P-93,S-T-91,Terescak94,W-Y-20,W-Y-22-2}) and measure-theoretic sense (\!\!\cite{E-H-S-08,WYZ-22,WYZ-24}).
Over the past decades, extensive research has been conducted in this field, and its applications have been continuously expanding. For more detailed information, one may refer to the monographs \cite{H-84,Smith,Hirsch-05,S-17,P-02,SY-98,Cheban-24} as well as their cited literature.

Numerous mathematical models, originating from differential and difference equations, can give rise to monotone dynamical systems (see \cite{Smith,S-17,Hirsch-05,Zhao-17}). Given that these models are frequently prone to noise-induced perturbations, the generated monotone dynamical systems are inevitably affected by a wide range of irregularly occurring phenomena. As a result, the study of how noise perturbations affect the dynamics of monotone systems has become a fundamental concern, critical to both theoretical modeling and the practical investigation of system behavior.

Prior research has addressed this dynamical issue, mainly concentrating on a trajectory-based approach that is predominantly utilized within the framework of random dynamical systems. One may refer to \cite{Arnold-Chueshov,Chueshov,Chueshov1,Flandoli} for recent developments in the theory of monotone (order-preserving) random systems, as well as the application to the long-time behavior of random and stochastic differential equations. For instance, Arnold and Chueshov \cite{Arnold-Chueshov} studied certain simplification in the long-time dynamics for random monotone systems. Synchronization by noise was subsequently observed in random monotone systems (see Flandoli et al.\cite{Flandoli}, Chueshov and Scheutzow et al.\cite{Chueshov1}), which indicates that there is a certain point attractor consisting of a single random point and in this sense the random dynamics are asymptotically globally stable. Thus, it becomes clear that the long-time behavior of a general monotone random systems is notably simpler when compared to its deterministic counterpart. In other words, the trajectory-based approach might not yield much insights into the noise-induced perturbations within monotone dynamical systems.

To our current understanding, in the realm of general noise perturbations, the {\it distribution-based approach} stands out as an especially potent strategy. It has become an increasingly essential instrument, providing greater utility than the trajectory-based method. A fundamental problem in the distribution-based approach is to identify the basic statistic dynamical units (for instance, certain invariant measures and their concentrations) which are {\it stochastically stable}, that is, they can ``survive" in the presence of noise interference. Following L. S. Young \cite{Young,Young1}, we called such invariant measures as {\it zero-noise limits of random perturbations of dynamical systems}, that is, zero-noise limits are the weak limits of $\mu^\varepsilon$ as $\varepsilon\to 0$. Here, $\mu^\varepsilon$ denotes the probability measure that is invariant under the Markov process corresponding to noise level $\varepsilon$. These ideas go back to Kolmogorov (see \cite[p.838]{Sinai}).

The importance of zero-noise limits is clear:
assuming that the real world is perpetually slightly noisy, they represent measures that are the truly observable invariant measures. Elucidating the precise locations for concentration of zero-noise limit is essential for the discernment of invariant sets exhibiting stochastic stability.
For some pioneer works from this point of view, we refer to \cite{Khasminskii,Nevel} for stochastic stability of flows on a cycle or a 2-torus, to \cite{Huang2016,Huang2018,Ruelle,Zeeman} for stochastic stability of (quasi)attractors, to \cite{Kifer,Kifer1} for stochastic stability of hyperbolic or axiom A systems, and to \cite{Young,Young2,Young1} for SRB measures as zero-noise limits with positive Lyapunov exponents, etc. Moreover, we refer to Freidlin and Wentzell \cite{FW1,FW2} for the stochastic stability of systems whose long-time behavior can be expressed as a finite union of certain basic sets, termed equivalence classes (with respect to the quasipotential) in their works.

This series of our work will adopt the distribution-based approach to study the impact of white noises on the dynamics of monotone dynamical systems. Since zero-noise limits mirror the stochastically stable behavior, one cannot solely rely on the analysis of the long-time behavior of individual orbits. Accordingly, due to its essential differences from deterministic monotone dynamical systems, much less is known in this direction comparing with cases using the trajectory-based approach.

In the present paper, we mainly focus on the zero-noise limits for the stochastic perturbation of a prototypical example of monotone dynamical systems, that is, the {\it cooperative} ordinary differential equations (see \cite{Hirsch}):
\begin{equation}\label{unpersys}
	\frac{\mathrm dx}{\mathrm dt}=b(x), \quad x(0)=x\in \mathbb{R}^r,
\end{equation}
where $b:\mathbb{R}^r\to \mathbb{R}^r$ is continuously differentiable. Cooperation means that an increase in one component of $x(t)$ causes an increase of all the other components, modeled by the assumption that $\partial b_i/\partial x_j\ge 0$ for $i\ne j$. Moreover, we further assume that the Jacobian matrices $Db(x)$ are {\it irreducible}, which means that each component of $x(t)$ directly or indirectly affect all the others. Such systems occur in many biological, chemical, physical and economic models.

For the stochastic perturbation of system \eqref{unpersys}, we consider the stochastic equations
\begin{equation}\label{itodff}
	\mathrm dX^{\varepsilon}_t=b(X^{\varepsilon}_t)\mathrm dt+\varepsilon\sigma(X^{\varepsilon}_t)\mathrm dW_t, \quad X^{\varepsilon}_0=x\in \mathbb{R}^r,
\end{equation}
where the perturbation parameter $\varepsilon$ is small, and $W_t=(W^{1}_t,\cdots,W^{r}_t)^T$ is a standard $r$-dimensional Wiener process.\! The diffusion matrix $\sigma$ is locally Lipschitz continuous and \textit{non-degenerate} in the sense that $a(x)\triangleq \sigma(x)\sigma^T(x)$ is positive definite for any $x\in \mathbb{R}^r$.
Here, $^{T}$ denotes transpose.

For a general vector field $b(x)$, many prior studies on the stochastic stability of system \eqref{unpersys} have primarily focused on investigating the asymptotic distribution of zero-noise limits near global/local attractors (or repellers) (\!\!\cite{Chen1,Chen2,Huang2018,XCJ,FW1,FW2,Huang2016,H-15,Ji1,Ji2,Ji3,H-15-2,H-15-3,H-15-4,L-Y-1,L-Y-2}), as well as their applications to various dynamical scenarios. Among them, Freidlin and Wentzell \cite{FW1,FW2} introduced the concept known as quasipotential, which is a natural extension of potential difference to non-equilibrium systems. Subsequently, they \cite{FW1,FW2} developed the large deviations theory, enabling the use of quasipotential to estimate the statistics of the likely transition paths during the analysis of rare events in stochastic dynamics. Utilizing this framework, they studied the equivalent classes with respect to the quasipotential, and established the stochastic stability of essentially a finite number of equilibria or periodic orbits. The common feature of all the above mentioned works is to deal with the stochastic stability of the pre-designated dynamical objectives, such as attractors/repellers, the equivalent classes (with respect to the quasipotential), etc.

Another different methodology of works has gone toward exploring the stochastic stability under assumptions about the large scale structure of the system. Such viewpoint has the advantage that it does not require a prior information to select special pre-designated dynamical objectives of interest. One may refer to the works in this direction for the gradient systems \cite{Huang2016,Hwang} (i.e., $b(x)$ is the negative of the gradient of a potential function), for axiom A systems \cite{Kifer,Young2}, and so on. However, monotone dynamical systems are typically neither gradient nor axiom A. Moreover, their long-term behavior may not be represented as a finite union of equivalent classes. This is due to the presence of a nonmonotone manifold (see Remark \hyperref[rmk-4.2]{4.2}), where the system can host infinitely many equivalent classes that exhibit complicated dynamics (c.f. \cite{S76,Smith1,Smith2,W-Y-22}).

We will make an attempt to contribute by using structural ideas to analyze the stochastic stability of monotone dynamical systems.
The basic theme of this paper is devoted to identifying the concentration of zero-noise limit of the cooperative and irreducible system (\ref{unpersys}). We will establish the stochastic stability of a dynamical order for the system \eqref{unpersys}, where the dynamical order is referred to a one dimensional simply ordered topological structure in the context (c.f. Chow et al. \cite{Chow}). This notion is also known as $p$-arcs, which was introduced in Mierczy\'{n}ski \cite{Janusz} (see also \cite{Hirsch1}). More precisely, the zero-noise limits will be concentrated on a simply ordered subset consisting of Lyapunov stable equilibria (see Theorem \ref{Mthm}). If, in addition, the system \eqref{unpersys} is either analytic or possesses only finitely many equilibria then any zero-noise limit is a convex combination of Dirac measures of a finite number of asymptotically stable equilibria (see Corollary \ref{Mthm-c} and Remark \hyperref[rmk-3.2]{3.2}). To our knowledge, this work seems to be the first investigation into stochastic stability for systems that may exhibit potentially infinite equivalence classes.

Our approach is motivated by the fusion of ideas from the work of Bena\"im and Hirsch \cite{Benaim} for stochastic approximation algorithms and the previous works of one of the present authors \cite{Chen2,XCJ}. According to large deviation theory, the transition dynamics is completely characterized by the quasipotential, whose landscape gives an intuitive description of the essential dynamical features. However, computing quasipotentials is very challenging, especially when the system is high dimensional, or when a global landscape is sought. This is exactly the situation we encounter when addressing the stochastic stability of monotone dynamical systems. To overcome such difficulties, one of our critical insights in the current work is to show that, for any unordered chain transitive set $K$ on a nonmonotone manifold (see Remark \hyperref[rmk-4.2]{4.2}), the quasipotential from $K$ to its upper (resp. lower) dual attractor consistently equals zero (see Proposition \ref{V-y-A=0}). Based on this, together with the arguments in \cite{XCJ}, we directly utilize the Freidlin--Wentzell theory (rather than the indirect discrete-time approach developed in \cite{Benaim}) to gauge the rare probability in the vicinity of unordered invariant sets; and hence, the stochastic stability of a dynamical order for the system \eqref{unpersys} is thus obtained.

Meanwhile, we further demonstrate our general framework by applying our results to the classical positive feedback models subjected to stochastic perturbation.
We show that their zero-noise limit is a convex combination of the Dirac measures on a finite number of asymptotically stable equilibria although such system may possess nontrivial periodic orbits.

This paper is organized as follows:
In Section \ref{preliminary}, we introduce the notations and preliminaries used throughout the paper.
Our main results will be presented in Section \ref{main-result}.
In Section \ref{quasi-mono}, we investigate the quasipotential for the monotone flow generated by system \eqref{unpersys}.
We give the proof of our main theorem in Section \ref{proof-of-main-results}.
Finally, in Section \ref{app}, we apply our main results to obtain the stochastic stability of the positive feedback systems.

\section{Notations and Preliminaries}\label{preliminary}

Let $(\mathbb{R}^r, \left|\cdot\right|)$ be the $r$-dimensional Euclidean space and $\langle\cdot,\cdot\rangle$ be the inner product in $\mathbb{R}^r$ defined as $\langle u,v \rangle=\sum_{i=1}^{r}u_iv_i$ for any $u,v\in\mathbb{R}^r$.
For each $x, y\in \mathbb{R}^r$, a partial order on $\mathbb{R}^r$ is given by $x\le y$ (resp. $x\ll y$) if and only if $x_i\le y_i$ (resp. $x_i< y_i$) for $i=1, \cdots, r$. We write $x<y$ if $x\le y$ and $x\neq y$.
A set $S$ is called {\it unordered} if no two of its points are related by $<$.
$S$ is called \textit{ordered} (resp. \textit{strongly ordered}) if any two points in $S$ are related by $``\leq"$ (resp. $``\ll"$).
In particular, a single point is referred to as ordered.
For $A,B\subset \mathbb{R}^r$, we write $A\leq B$ (resp. $A<B$, $A\ll B$) if $a\leq b$ (resp. $a<b$, $a\ll b$) for all $a\in A$ and $b\in B$.
Given $x,y\in\mathbb{R}^r$, the set $[[x, y]] = \{z\in\mathbb{R}^r:x\ll z\ll y\}$ is called an \textit{open order interval}, and we write \textit{closed order interval} $[x, y] =\{z\in\mathbb{R}^r:x\leq z\leq y\}$. In
particular, we write $[[x,+\infty]]=\{y\in\mathbb{R}^r: x\ll y\}$ and $[[-\infty,x]]=\{y\in\mathbb{R}^r: y\ll x\}$.
The \textit{supremum} $\sup S$ of a subset $S\subset\mathbb{R}^r$, if it exists, is the minimal point $a$ such that $a\geq S$. The \textit{infimum} $\inf S$ is defined dually.

The \textit{solution flow} of deterministic system (\ref{unpersys}) starting at $x$ is denoted by $\Phi_{t}(x)$. Throughout this paper, we always assume that system \eqref{unpersys} is forward complete, that is, the domain of any of its nonextendible solution flow $\Phi_t$ contains $[0,+\infty)$.
The \textit{positive orbit} of $x$ is denoted by $\mathcal{O}^+(x) = \left\{\Phi_{t}(x):t\geq0\right\}$.
An \textit{equilibrium} $p\in\mathbb{R}^r$ is a point for which $\mathcal{O}^+(p)=\left\{p\right\}$.
The set of all equilibria of \eqref{unpersys} is denoted by $\mathcal{E}$.
An equilibrium $p\in\mathcal{E}$ is called \textit{Lyapunov stable} if for each neighborhood $V$ of $p$ there exists a neighborhood $V_{1}\subset V$ of $p$ such that $\Phi_{t}V_1\subset V$ for all $t \geq 0$;
and called \textit{asymptotically stable} if there exists a neighborhood $N$ of $p$ such that $\lim_{t\to\infty}{\rm dist}(\Phi_t(x),p)=0$ uniformly for $x\in N$.
The $\omega$-limit set $\omega(x)$ of $x$ is defined by $\omega(x)=\bigcap_{\tau\geq0}\overline{\mathcal{O}^+(\Phi_\tau(x))}$, where the closure (resp. boundary) of a set $S\subset\mathbb{R}^r$ is denoted by $\overline{S}$ (resp. $\partial S$).
A set $D\subset\mathbb{R}^r$ is {\it invariant} (resp. {\it positively invariant}) if $\Phi_{t}D= D$ (resp. $\Phi_{t}D\subset D$) for all $t\geq 0$.
A non-empty, compact and invariant subset $\mathcal{A}$ is called to be an {\it attractor} if there is a {\it fundamental neighborhood} $N$ of $\mathcal{A}$ such that
$\displaystyle\lim_{t\rightarrow \infty}
{\rm dist}\big(\Phi_{t}(x),\mathcal{A}\big)=0$ uniformly in $x\in N$. The \textit{basin of attraction} of $\mathcal{A}$ is the open set
$${\rm Basin}(\mathcal{A})=\Big\{x\in\mathbb{R}^r:\lim_{t\to\infty}{\rm dist}(\Phi_t(x),\mathcal{A})=0\Big\}.$$
The system \eqref{unpersys} is called \textit{dissipative} if there exists an attractor $\Lambda$ whose basin of attraction is the space $\mathbb{R}^r$ (where $\Lambda$ is referred to as the global attractor).
For $x\in\mathbb{R}^r,\delta>0$ and $U\subset\mathbb{R}^r$, we write $B_\delta(x)=\{y\in\mathbb{R}^r:|y-x|<\delta\}$
and $U_{\delta}=\{y\in\mathbb{R}^r:{\rm dist}(y,U)<\delta\}$, where ${\rm dist}(y,U):=\inf_{x\in U}|x-y|$.

For $u,v\in\mathbb{R}^r$, we say \textit{$u$ chains to $v$}, written $u\overset{\Phi}{\rightsquigarrow} v$, if for every $T>0$ and $\varepsilon> 0$ there exist an integer $n>0$, $t_1,\cdots,t_n\geq T$, and a finite sequence in $\mathbb{R}^r$ of the form $\{u=y_0,\cdots,y_n=v\}$ such that $|\Phi_{t_i}(y_{i-1})-y_{i}|<\varepsilon$ for $i= 1,\cdots,n$.
If $u\overset{\Phi}{\rightsquigarrow} v$ and $v\overset{\Phi}{\rightsquigarrow} u$ then $u$ and $v$ are \textit{chain equivalent}, written $u\overset{\Phi}{\approx} v$.
A compact invariant set $K$ is called \textit{chain transitive} (resp. \textit{externally chain transitive}) if for any two points $u,v\in K$ one has $u\overset{\Phi|_K}{\approx} v$ (resp. $u\overset{\Phi}{\approx} v$).

By a \textit{$p$-arc} (see Mierczy\'{n}ski \cite[p.1480]{Janusz}), we mean a strongly ordered invariant set $J \subset \mathbb{R}^r$ such that there exists an increasing $C^1$-smooth function $h$ from a compact interval $I := [0, 1] \subset \mathbb{R}$ onto $J$, where $h$ is increasing if $h(s_1) \ll h(s_2)$ whenever $0 \leq s_1 < s_2 \leq 1$. In addition, if $J\subset \mathcal{E}$, then $J$ is called a \textit{stationary $p$-arc}.

A flow $\Phi_t$ is called \textit{monotone} if $\Phi_t(x) \leq \Phi_t(y)$ provided $x\leq y$ for any $t >0$; and called \textit{strongly monotone} if $\Phi_t(x)\ll\Phi_t(y)$ provided that $x<y$ for all $t > 0$.
As one may know, the solution flow $\Phi_t$ of the cooperative and irreducible system \eqref{unpersys} is strongly monotone (see \cite{Smith,Hirsch-05}). Here, system \eqref{unpersys} is called \textit{cooperative} if $\partial b_i/\partial x_j\ge 0$ for $i\ne j$, and \textit{irreducible} if its Jacobian matrix \( Db(x) \) is irreducible for each \( x \in \mathbb{R}^r \).	Recall that a \( (r \times r) \) matrix \( A = (a_{ij}) \) is irreducible if for every nonempty, proper subset \( I \) of the set \( N = \{1,2,\ldots,r\} \), there is an \( i \in I \) and \( j \in N\backslash I \) such that \( a_{ij} \neq 0 \).

\vspace{2ex}
We now present some notations in the large deviation theory, which are useful in dealing with stationary measure asymptotics for small noise Markov processes (see, e.g., \cite{FW1, FW2, Kifer} and the references therein).

Let us consider a complete probability space $(\Omega,\mathcal{F},\{\mathcal{F}_t\}_{t\geq0},\mathbb{P})$ endowed with a filtration $\{\mathcal{F}_t\}_{t\geq0}$ that satisfies the usual conditions.
Denote by $\mathcal{P}(\mathbb{R}^r)$ the set of all probability measures on $\mathbb{R}^r$.
Let $\mu\in \mathcal{P}(\mathbb{R}^r)$, we say that a sequence $\{\mu_n\}\subset \mathcal{P}(\mathbb{R}^r)$ \textit{converges weakly} to $\mu$ if for any bounded continuous function $f$ on $\mathbb{R}^r$ one has $\lim_{n\to\infty}\int_{\mathbb{R}^r} fd\mu_n=\int_{\mathbb{R}^r} fd\mu.$
We say that $\Pi\subset \mathcal{P}(\mathbb{R}^r)$ is \textit{tight} if for each $\eta > 0$ there exists a compact set $K\subset\mathbb{R}^r$ such that $\mu(K) > 1 -\eta$ for every $\mu\in\Pi$.
The \textit{support} of $\mu\in\mathcal{P}(\mathbb{R}^r)$, denoted by ${\rm supp}(\mu)$, is the smallest closed set whose complement has measure $0$ under $\mu$.

Fix $T>0$,
let ${\bf C}_{T}=C([0,T],\mathbb{R}^r)$ (resp. ${\bf AC}_T=AC([0,T],\mathbb{R}^r)$) denote the set of continuous (resp. absolutely continuous) functions on $[0,T]$ with values in $\mathbb{R}^r$.
For $\varphi,\psi\in{\bf C}_{T}$ and $W\subset {\bf C}_{T}$, define $\rho_{T}(\varphi,\psi)=\sup\limits_{0\leq t\leq T}|\varphi(t)-\psi(t)|$ and $\rho_{T}(\varphi,W)=\inf\limits_{\phi\in W}\rho_{T}(\varphi,\phi)$. Clearly, $(\textbf{C}_T, \rho_T)$ is a complete metric space.
For system \eqref{itodff}, we define the \textit{rate energy} on ${\bf C}_{T}$:
\begin{equation*}\label{actfal}
	\mathcal{S}_{T}(\varphi)=\left\{
	\begin{array}{ll}
		\int_{0}^{T}L(\varphi(t),\dot{\varphi}(t)){\rm d}t, & \hbox{if $\varphi\in{\bf AC}_T$,} \vspace{2mm}\\
		\infty, & \hbox{otherwise,}
	\end{array}
	\right.
\end{equation*}
where
\begin{equation*}
L(u,\beta)=\frac{1}{2}\big(\beta-b(u)\big)^T a^{-1}(u)\big(\beta-b(u)\big)\quad \text{ for any } u,\beta\in\mathbb{R}^r,
\end{equation*}
and $\dot{\varphi}$ means the derivative of $\varphi$.
By virtue of lemma due to Riesz and Nagy (c.f. \cite[Chapter 41]{R-SN}),
$\mathcal{S}_T$ is \textit{lower semi-continuous} in ${\bf C}_T$ for any $T\!>\!0$ in the sense that $\liminf_{\rho_T\!(\varphi_n,\varphi)\!\to\!0}\!\mathcal{S}_T(\varphi_n)\!\geq\!\mathcal{S}_T(\varphi)$ for any $\varphi\in {\bf C}_T$ (see \cite[p.61, Lemma 3.2.1(a)]{FW2}).
For each $H\subset \textbf{C}_T$, let $\mathcal{S}_T(H):=\inf\{\mathcal{S}_T(\varphi): \varphi\in H\}$ and $\mathcal{S}_T(\emptyset)=\infty$.
To emphasize the dependence of initial conditions $x$, we also introduce the notions
${\bf C}_{T}^{x}=\{\varphi\in{\bf C}_{T}:\varphi(0)=x\}$,
${\bf AC}_T^x=
\{\varphi\in{\bf AC}_T:\varphi(0)=x\}$ and the rate
function $\mathcal{S}_{T}^{x}$ on ${\bf C}_{T}^x$:
\begin{equation*}\label{actfalini}
	\mathcal{S}_{T}^{x}(\varphi)=\left\{
	\begin{array}{ll}
		\int_{0}^{T}L(\varphi(t),\dot{\varphi}(t)) {\rm d}t, & \hbox{if $\varphi\in{\bf AC}_T^x$,} \vspace{2mm}\\
		\infty, & \hbox{otherwise.}
	\end{array}
	\right.
\end{equation*}
To focus on those rare energy functions, we write \textit{the level set $\mathbb{F}_{T}^{x}(s)$} of $\mathcal{S}_{T}^{x}$ as
\begin{equation}
\mathbb{F}_{T}^{x}(s)=\{\varphi\in{\bf C}_{T}^x:\mathcal{S}_{T}^{x}(\varphi)\leq s\},\quad\text{ for }T>0, x\in\mathbb{R}^r \text{ and } s\geq 0.\tag*{}
\end{equation}

\phantomsection
\noindent\textbf{Remark 2.1.}\label{rmk-2.1} $\mathcal{S}_{T}^{x}(\varphi)=0$ if and only if $\varphi$ (up to time $T$) coincides with the solution of the deterministic system {\rm(\ref{unpersys})}.

\vspace{1ex}
In the following, we introduce the \textit{Linear Interpolation Function} of $x,y\in\mathbb{R}^r$ (abbr. ${\rm LIF}_{x,y}$) as
\begin{equation*}
	{\rm LIF}_{x,y}(t)=x+\frac{t}{|y-x|}(y-x),\quad\quad t\in\big[0,|y-x|\big].
\end{equation*}

\vspace{1ex}
We have the following useful lemma.
\begin{lem}\label{Vctin}
For each compact set $K\subset \mathbb{R}^r$, there is a $L>0$ such that
\begin{equation*}
\mathcal{S}^x_{|x-y|}({\rm LIF}_{x,y})\leq L|x-y|\quad\text{ for any }x,y\in K.
\end{equation*}
\end{lem}
\begin{proof}
	See \cite[Lemma 4.2.3]{FW2}.
\end{proof}

\vspace{1ex}
In the following, we present the hypothesis of the Freidlin--Wentzell uniform large deviations principle, which can be found in the literature (see, e.g.,\cite{FW1,FW2,DZ} and the references therein).
Let $\left\{X^{\varepsilon,x}_{\cdot}\right\}$, parameterized by $x\in\mathbb{R}^r$ and $\varepsilon>0$, be the solution of {\rm (\ref{itodff})}.
Let $\mathcal{K}$ be a collection of all compact subsets of $\mathbb{R}^r$.

\begin{defn}[Freidlin--Wentzell uniform large deviations principle over $\mathcal{K}$]\label{def-F-W-LDP}
Let $T>0$.
$\{X^{\varepsilon,x}_\cdot\}$ is said to satisfy a \textit{Freidlin--Wentzell uniform large deviations principle} with respect to the rate functions $\mathcal{S}_{T}^{x}$ uniformly over $\mathcal{K}$, if
\begin{enumerate}[\bf(C)]

\item\phantomsection\label{Compactness} (Compactness): For each $s<\infty$ and $K\in\mathcal{K}$, the set $\bigcup_{x\in K}\mathbb{F}_{T}^{x}(s)$ is compact in $\textbf{C}_T$;

\item[\bf(L)]\phantomsection\label{Lower} (Lower bound) : For each $\delta>0$, $\gamma>0$, $s_0>0$ and $K\in\mathcal{K}$, there exists $\varepsilon_0>0$ such that
\begin{equation}\label{ufldplbb}
	\mathbb{P}\{\rho_{T}(X^{\varepsilon,x}_\cdot,\varphi)<\delta\}\geq \exp\left\{-\frac{\mathcal{S}^x_{T}(\varphi)+\gamma}{\varepsilon^2}\right\}
\end{equation}
\noindent for any $\varepsilon\in(0,\varepsilon_0]$ and $\varphi\in\mathbb{F}_{T}^{x}(s_0)$ with $x\in K$;

\item[\bf(U)]\phantomsection\label{Upper} (Upper bound): For each $\delta>0$, $\gamma>0$, $s_0>0$ and $K\in\mathcal{K}$, there exists $\varepsilon_0>0$
such that
\begin{equation}\label{ufldpupbb}
	\mathbb{P}\{\rho_{T}(X^{\varepsilon,x}_\cdot,\mathbb{F}_T^x(s))\geq\delta\}\leq \exp\left\{-\frac{s-\gamma}{\varepsilon^2}\right\}
\end{equation}
for any $\varepsilon\in(0,\varepsilon_0]$, $s\in[0,s_0]$ and $x\in K$.
\end{enumerate}
\end{defn}

\vspace{2ex}
A straight forward consequence of (\hyperref[Compactness]{C}) is the following

\begin{lem}\label{p-64}
Assume that {\rm(\hyperref[Compactness]{C})} holds. For any $T > 0$, let $H\subset{\bf C}_T$ be a closed set and $H_0:= \{\varphi(0):\varphi\in H\}$ a compact set in $\mathbb{R}^r$. Then $\mathcal{S}_T(H)>0$ whenever $H$ does not contain any solution of system \eqref{unpersys}.
\end{lem}
\begin{proof}
Suppose that $\mathcal{S}_T(H)=0$.
Then, one can choose a sequence $\{\varphi_n\}_{n\geq1}\subset H$ such that
\begin{equation}\label{S-phi-0}
\lim_{n\to\infty}\mathcal{S}_T(\varphi_n)=0.
\end{equation}
Given $s>0$, without lost of generality, we assume $\{\varphi_n\}$ is such that
$\mathcal{S}_T(\varphi_n)<s$ for any $n\geq1$, that is,
\begin{equation}\label{S-phi-1}
	\{\varphi_n\}\subset \bigcup_{x\in H_0}\mathbb{F}^x_T(s).
\end{equation}
By (\hyperref[Compactness]{C}), the set $F\triangleq\left(\bigcup_{x\in H_0}\mathbb{F}^x_T(s)\right)\cap H$ is compact in ${\bf C}_T$.
It then follows that there exists a subsequence of $\{\varphi_n\}_{n\geq1}$, still denoted by $\{\varphi_n\}$, converges to some $\varphi_0\in F$.

We show that $\mathcal{S}_T(\varphi_0)=0$.
In fact, \eqref{S-phi-0} and the lower semi-continuity of $\mathcal{S}_T(\varphi)$ with respect to $\varphi\in{\bf C}_T$ (see \cite[Lemma 3.2.1(a)]{FW2}) imply that
\begin{equation*}
\mathcal{S}_T(\varphi_0)\leq\lim_{n\to\infty}\mathcal{S}_T(\varphi_n)=0.
\end{equation*}
Together with Remark \hyperref[rmk-2.1]{2.1}, this entails that $\varphi_0\in F\subset H$ coincides with the solution of system \eqref{unpersys} up to time $T$, a contradiction.
\end{proof}

\section{Main Results}\label{main-result}

Before presenting our main results, we make the following assumptions:

\vspace{2.5ex}
\phantomsection
\noindent\textbf{(H1)}\label{H-1} System \eqref{unpersys} is cooperative, irreducible and dissipative on  $\mathbb{R}^r$.
\vspace{2ex}

\phantomsection
\noindent\textbf{(H2)}\label{H-2} There exists a nonnegative smooth function $V_1$ on $\mathbb{R}^r$ and positive numbers $\gamma,\varepsilon_0, R>0$ such that
\begin{equation}\label{Vinfty}
	\lim_{|x| \to \infty} V_1(x) = \infty
\end{equation}
and
\begin{equation}\label{Vdissip}
	\langle b(x), \nabla V_1(x) \rangle + \frac{\varepsilon^2}{2} \text{Tr} \left( \sigma^T(x) D^2 V_1(x) \sigma(x) \right) \leq -\gamma
\end{equation}
for any $\varepsilon \in [0, \varepsilon_0]$ and $|x| \geq R$, where $\text{Tr}(A)$ denotes the trace of the matrix $A$, $\nabla V_1(x)$ represents the gradient of $V_1(x)$ and $D^2V_1(x)$ stands for Hessian of $V_1(x)$.

\vspace{2ex}
\phantomsection
\noindent\textbf{(H3)}\label{H-3} There exist a nonnegative smooth function $V_2$ on $\mathbb{R}^r$ and positive numbers $\theta,\eta,C,M>0$ such that
\begin{equation}\label{21}
	\lim_{|x| \to \infty} V_2(x) = \infty
\end{equation}
and
\begin{equation}\label{22}
	\langle b(x),\nabla V_2(x)\rangle+\frac{\theta}{2}{\rm Tr}\left(\sigma^T(x)D^2 V_2(x)\sigma(x)\right)+\frac{|\sigma^T(x) \cdot \nabla V_2 (x)|^2}{\eta V_2(x)}\leq C\big(1+V_2(x)\big)
\end{equation}
and
\begin{equation}\label{23}
	{\rm Tr}\left(\sigma^T(x)D^2 V_2(x)\sigma(x)\right)\geq -M -CV_2(x)
\end{equation}
for any $x\in\mathbb{R}^r$.

\vspace{2ex}
\noindent\textbf{Remark 3.1.}
More relevant critical information from (\hyperref[H-1]{H1})-(\hyperref[H-3]{H3}) is as follows:
\begin{enumerate}[(i)]
\item (\hyperref[H-2]{H2}) guarantees the following key facts:
\begin{enumerate}[(a)]
	\item System \eqref{itodff} admits a unique stationary measure $\mu^{\varepsilon}$ for each $\varepsilon>0$ sufficiently small (see, e.g., Khasminskii \cite[Theorem 4.1 and Corollary 4.4]{Kha-80} or Huang et al. \cite[Theorem A]{H-15-2});
	\item The family $\mathscr{I}:=\{\mu^{\varepsilon} : \varepsilon>0 \text{ sufficiently small}\}$ is tight, i.e., for each $\eta > 0$ there exists a compact set $K\subset\mathbb{R}^r$ such that $\nu(K) > 1 -\eta$ for every $\nu\in\mathscr{I}$ (see, e.g., Huang et al. \cite[Theorem B and Remark 2.3(2)]{Huang2018});
\end{enumerate}
\item (\hyperref[H-3]{H3}) guarantees that the solution $\left\{X^{\varepsilon,x}_{\cdot}\right\}$ admits the Freidlin--Wentzell uniform large deviations principle with respect to the rate functions $\mathcal{S}_{T}^{x}$ (see \cite[Theorem 2.1]{JWZZ}).
\item The irreducible assumption in (\hyperref[H-1]{H1}) is only to ensure the strong monotonicity of the system. In fact, the irreducibility can be weakened as in Section \ref{app} to guarantee the strong monotonicity.
\end{enumerate}

\vspace{1ex}
Our main theorem is following:

\begin{thm}\label{Mthm}
Assume that {\rm(\hyperref[H-1]{H1})}-{\rm(\hyperref[H-3]{H3})} hold.
Let $\mu$ be a zero-noise limit of the system \eqref{itodff} and $H$ be any connected component of ${\rm supp}(\mu)$. Then $H$ is contained in an arc {\rm(}possibly degenerate{\rm)} of Lyapunov stable equilibria of system \eqref{unpersys}.
\end{thm}

For the cooperative and irreducible system \eqref{unpersys}, the concentration of the zero-noise limit $\mu$ on the stable stationary $p$-arc indicates the stochastic stability of \textit{a dynamical order}, which is referred as certain one-dimensional simply ordered topological structure (see \cite{Chow}).

In particular, if the system \eqref{unpersys} is analytic, then such dynamical order will degenerate into a finite number of asymptotically stable equilibria (see \cite[Theorem 3]{Jiang1} and \cite[Theorem 2]{Jiang2}), by which one can immediately obtain the following corollary:

\vspace{1ex}
\begin{cor}\label{Mthm-c}
Let all the hypotheses in Theorem \ref{Mthm} hold.
Assume further that system \eqref{unpersys} is analytic.
Then the zero-noise limit $\mu$ satisfies
\begin{equation*}
\mu=\sum_{i=1}^n\lambda_i\delta_{E_i}(\cdot)\quad \text{ with } \sum_{i=1}^n\lambda_i=1,
\end{equation*}
where $E_i$ is an asymptotically stable equilibrium of \eqref{unpersys}, and $\delta_{E_i}(\cdot)$ is the Dirac measure on $E_i$, for $i=1,\cdots,n$.
\end{cor}

\vspace{1ex}
\phantomsection
\noindent\textbf{Remark 3.2.}\label{rmk-3.2} If system \eqref{unpersys} has a finite number of stable equilibria, the zero-noise limit is a convex combination of Dirac measures on these equilibria.

\section{Quasipotential for Monotone Flows}\label{quasi-mono}

In order to establish our main results, we first investigate the characteristic of the quasipotential for monotone flows $\Phi_t$ generated by system \eqref{unpersys}.
Throughout this section, we always assume (\hyperref[H-1]{H1}) holds.

The concept of quasipotential (see, e.g., \cite [p.90]{FW2}) is crucial for understanding the mechanisms of rare events and for characterizing the statistical behavior of transitions between stable and unstable states in stochastic dynamics.
Given $x, y\in\mathbb{R}^r$, the \textit{quasipotential} from $x$ to $y$ is defined by
$$V(x,y)\triangleq\inf_{T>0}\inf_{\varphi\in \textbf{AC}^x_T}\big\{\mathcal{S}^x_{T}(\varphi):\varphi(0)=x,\ \varphi(T)=y\big\}.$$
Without ambiguity, for a pairs of subsets $D_0,D_1\subset\mathbb{R}^r$, we define the quasipotential from $D_0$ to $D_1$ by
$$V(D_0,D_1)\triangleq\inf_{T>0}\inf_{\varphi\in \textbf{AC}_T}\big\{\mathcal{S}_{T}(\varphi):\varphi(0)\in D_0,\varphi(T)\in D_1\big\}.$$

To identify the concentration of zero-noise limits, it pays to be on the lookout of the recurrence and statistical behavior of orbits for the monotone system \eqref{unpersys}.
Particularly, the support ${\rm supp}(\mu)$ of the zero-noise limit is strongly related to the sets of recurrent points (see \cite{Chen2}), or more generally, the chain-transitive sets.

The following order-structure dichotomy for chain-transitive sets is due to Hirsch \cite{Hirsch1} (see also Mierczy\'{n}ski \cite{Janusz1} for externally chain-transitive sets):

\begin{lem}\label{Hirsch_99}
	Let $K \subset \mathbb{R}^r$ be a chain-transitive set for the flow $\Phi_t$. Then either $K$ is unordered, or $K$ is a stationary $p$-arc.
\end{lem}

\vspace{1ex}
\noindent\textbf{Remark 4.1.}
By a result of Conley \cite{Conley}, a compact invariant set is chain-transitive if and only if it is attractor-free. Here, $K\subset\mathbb{R}^r$ is called \textit{attractor-free} if the restricted flow $\Phi_t|_K$ admits no attractor other than $K$ itself.

\vspace{2ex}
In the following, we will investigate the quasipotential from the chain-transitive set to certain dual attractors.
The following proposition describes the quasipotential from unordered chain-transitive sets.

\begin{prop}[Quasipotential from unordered chain-transitive sets]\label{V-y-A=0}
Let $K\subset\mathbb{R}^r$ be an unordered chain-transitive set.
Then, there exists an attractor $\mathcal{A}=\mathcal{A}(K)$ with $K\cap\mathcal{A}=\emptyset$ such that
\begin{equation}\label{VyA=0}
	V(y,\mathcal{A})=0 \quad\text{ for any }y\in K.
\end{equation}
\end{prop}

\begin{proof}
Let $p=\sup K$. Clearly, $K<p$, since $K$ is unordered.
Due to the invariance of $K$ and the strong monotonicity, one has
$p\ll\Phi_t(p)$ for any $t > 0$; and hence,
$\Phi_t[p,+\infty]]\subset[[p,+\infty]]$ for any $t > 0$.
Let
\begin{equation}\label{def-A}
	\mathcal{A}=\Lambda\cap\bigcap_{t>0}\Phi_t[p,+\infty]],
\end{equation}
where $\Lambda$ denotes the global attractor under the dissipation assumption in (\hyperref[H-1]{H1}). Then,
$\mathcal{A}$ is an attractor satisfying $K\ll\mathcal{A}$ and $K\cap\mathcal{A}=\emptyset$.
We call $\mathcal{A}$ the \textit{upper dual attractor} with respect to $K$.
Moreover, define
$$H^+(K)=\{x\in \mathbb{R}^r:\ y\ll\Phi_s(x) \text{ for some }y\in K\text{ and } s\ge 0\}.$$
We assert that
\begin{equation}\label{Hk-basin}
H^+(K)\subset{\rm Basin}(\mathcal{A}).
\end{equation}
For this purpose, it suffices to show the fact that if $y\ll x$ for some $y\in K$, then $p< \omega(x)$ (Indeed, if $x\in H^+(K)$, then this fact, together with the invariance of $\omega(x)$, indicates that $p<\omega(\Phi_s(x))=\omega(x)$ for some $s\geq0$, which means that $x\in {\rm Basin}(\mathcal{A})$ (see \eqref{def-A})).
Now, given $y\in K$ with $y\ll x$. By virtue of the Limit Set Dichotomy (see \cite[Theorem 2.4.5]{Smith}), one has either Case (i): $\omega(y)\ll\omega(x)$; or otherwise Case (ii): $\omega(y)=\omega(x)=\{e\}$ for some $e\in\mathcal{E}$.
We will deal with these two cases, respectively.

Case (i): $\omega(y)\ll\omega(x)$.
Let $$W=\{z\in K: z\ll \omega(x)\}.$$
Clearly, $\omega(y)\subset W\neq\emptyset$.
Due to the invariance of $\omega(x)$ and strong monotonicity in (\hyperref[H-1]{H1}), one has
$\Phi_t\overline{W}\subset W$ for any $t>0$.
This implies that there exists an attractor in $W$ with respect to $\Phi_t|_K$.
Noticing that $K$ is attractor-free, we obtain that $K=W$, which implies $K\ll  \omega(x)$.
Therefore, $p<\omega(x)$.

Case (ii): $\omega(y)=\omega(x)=\{e\}$ for some $e\in\mathcal{E}$.
Clearly, $e\in K$ (because $y\in K$). Let
$\lambda_1(e)$ be the largest real part of the eigenvalues of the linearization matrix $Db(e)$ of the vector field $b(\cdot)$ at $e$. 

We assert that $\lambda_1(e)\leq0$.
The proof of the assertion is motivated by \cite[Theorem 2.4.5]{Smith}. For completeness, we give the detail.
Suppose that \( \lambda_1(e) > 0 \). Then, the spectral radius $\rho(e)$ of $D\Phi_1(e)$ satisfies \( \rho(e) > 1 \).
Hence, \cite[Lemma 2.4.2]{Smith} implies that there exists a unit eigenvector \( z_0=z_0(e) \gg 0 \) such that
\begin{equation}\label{eig-z0}
	D\Phi_1(e)z_0=\rho(e)z_0.
\end{equation}
For each $n\geq1$, let \( u_n = \Phi_n(x) \) and \( v_n = \Phi_n(y) \). By the strong monotonicity of \( \Phi_t \), \( v_n \ll u_n \). Define $\alpha_n \triangleq \sup\left\{ \alpha > 0 : v_n + \alpha z_0 \leq u_n \right\} > 0.$
So, $v_n\ll v_n + \alpha_n z_0 \leq u_n$, for each $n\geq1.$
Note also that \( \omega(x) = \omega(y) = \{e\} \). Then \( u_n \to e \) and \( v_n \to e \) as \( n \to \infty \), which implies $\alpha_n\to0$ as $n\to\infty$. On the other hand, for each $n\geq1$, one has
\begin{equation}\label{Sl}
\begin{split}
\Phi_1(v_n + \alpha_n z_0) - \Phi_1(v_n) = D\Phi_1(e)\alpha_n z_0 + \alpha_n \delta_n
\xlongequal{\eqref{eig-z0}}\rho(e)\alpha_n z_0 + \alpha_n \delta_n,
\end{split}
\end{equation}
where
$\delta_n = \int_{0}^{1} \left( D\Phi_1(v_n + s\alpha_n z_0) - D\Phi_1(e) \right) z_0 ds.$
Since \( D\Phi_1(z) \) is continuous with respect to \( z \), and \( v_n + s\alpha_n z_0 \to e \) as $n\to\infty$ uniformly for \( s \in [0,1] \),  we obtain
$\sup_{0 \leq s \leq 1}|( D\Phi_1(v_n + s\alpha_n z_0) - D\Phi_1(e))z_0| \to 0,$ as $n \to \infty.$
This implies \( \left|\delta_n\right| \to 0 \) as \( n \to \infty \). Let \( d_n \triangleq \inf\left\{ \beta > 0 : -\beta z_0 \leq \delta_n \leq \beta z_0 \right\} \). Then, $d_n\to0$ as $n\to\infty$. Consequently, $\rho(e) - d_n > 1$ for $n$ sufficiently large.
Noticing that $\delta_n\geq- d_n z_0$, one can substitute this into \eqref{Sl}, which yields
\begin{equation*}
\Phi_1(v_n + \alpha_n z_0) - \Phi_1(v_n)\geq \big(\rho(e)-d_n\big)\alpha_n z_0 >\alpha_n z_0,
\end{equation*}
for all $n$ sufficiently large. In other words,
$\Phi_1(v_n + \alpha_n z_0) > \Phi_1(v_n) + \alpha_n z_0 = v_{n+1} + \alpha_n z_0$.
Hence, $u_{n+1} > v_{n+1} + \alpha_n z_0,$ for all $n$ sufficiently large
(due to monotonicity of \( \Phi_1 \) and \( v_n + \alpha_n z_0 \leq u_n \)).
So, by the definition of \( \alpha_{n+1} \), it entails that \( \alpha_{n+1} \geq \alpha_n \), contradicting $\alpha_n \to 0$. Thus, we have proved the assertion.

Due to the assertion $\lambda_1(e)\leq0$, it follows from Mierczy\'{n}ski \cite[Proposition 1.3(i) and (iii)]{Janusz} that there exists an invariant set $C_e$ which is a one-codimensional sub-manifold homeomorphic to some open set in $\mathbb{R}^{r-1}$ satisfying
\begin{equation}\label{prop-c}
|\Phi_t(z)- e|\to0\ \ \text{ as }\ t\to\infty,\ \quad\text{ for any }z\in C_e.
\end{equation}
Fix $v\gg0$. For such $C_e$, it follows from Hirsch \cite[Proposition 2.7]{H-88} that the map $F: C_e\times \mathbb{R}\to \mathbb{R}^r$; $(z,\lambda)\mapsto z+\lambda v$
is a homeomorphism onto an open subset $U$ in $\mathbb{R}^r$.
Clearly, $e\in U$. Then, one can find a small $\delta>0$ such that $B_\delta(e)\subset U$; and hence, any \(w \in B_\delta(e) \) can be written as \( w = z + \lambda_w v \), where \( z \in C_e \) and \( \lambda_w \in \mathbb{R} \).

Let $K^\delta=K\cap B_\delta(e)$. We \textit{claim that} $K^\delta\subset C_e$.
Otherwise, choose a $w\in K^\delta\backslash C_e\neq\emptyset$.
Since $w\in B_\delta(e)$, there exists $z\in C_e$ and $\lambda_w\in\mathbb{R}$ such that $w=z+\lambda_wv$, which implies that $w$ and $z$ are related by $``\ll"$.
By virtue of \eqref{prop-c} and strong monotonicity, $\omega(w)$ and $\{e\}$ are related by $``\le"$.
Noticing that $e\in K$, $w\in K$ (hence $\omega(w)\subset K$) and $K$ is unordered, we obtain $\omega(w)=\{e\}$; and hence, $w\in{\rm Basin}(e)\backslash C_e.$
By virtue of \cite[Proposition 1.3(v)]{Janusz}, $\Phi_t(w)$ converges monotonically to $e$ as $t\to\infty$.
This leads to a contradiction, since $w\in K$ and $K$ is unordered. Thus, we have proved the claim.

Due to \eqref{prop-c}, the claim then implies that $e$ is an attractor in $K$, contradicting that $K$ is attractor-free. So, case (ii) cannot occur. Thus, we have proved the \eqref{Hk-basin}.

\vspace{2ex}

Now, choose $\delta_0>0$ so small that $\overline{K_{\delta_0}}\cap\overline{\mathcal{A}_{\delta_0}}=\emptyset$.
In order to prove \eqref{VyA=0} for each $y\in K$, one needs to show that, for any $\eta>0$, there exist $T>0$ and $\psi^y\in {\bf AC}_T^y$ with $\psi^y(0)=y$, $\psi^y(T)\in \mathcal{A}$ such that
\begin{equation*}
	\mathcal{S}^y_{T}(\psi^y)\leq \eta.
\end{equation*}
To this end, let $L_1,L_2>0$ be given in Lemma \ref{Vctin} for the compact sets $\overline{K_{\delta_0}}$ and $\overline{\mathcal{A}_{\delta_0}}$, respectively, such that
\begin{equation}\label{SLIF_le_L_K}
	\mathcal{S}^{x_1}_{|x_1-x_2|}({\rm LIF}_{x_1,x_2})\leq L_1|x_1-x_2|\quad \text{ for any } x_1,x_2\in \overline{K_{\delta_0}},
\end{equation}
and
\begin{equation}\label{SLIF_le_L_A}
	\mathcal{S}^{x_1}_{|x_1-x_2|}({\rm LIF}_{x_1,x_2})\leq L_2|x_1-x_2|\quad \text{ for any } x_1,x_2\in \overline{\mathcal{A}_{\delta_0}}.
\end{equation}
Let $\delta=\min \left\{\delta_0,\dfrac{\eta}{L_1+L_2}\right\}>0$ and $y_{\delta}=y+\delta v$ with $v\gg 0$, $|v|=1$.
Clearly,
\begin{equation}\label{delta_1}
	y\ll y_\delta\quad\text{ and }\quad|y-y_\delta|= \delta.
\end{equation}
Moreover,
$y_{\delta}\in H^+(K)$.
By virtue of \eqref{Hk-basin},
$$\lim_{t\rightarrow \infty}{\rm dist}(\Phi_t(y_{\delta}), \mathcal{A})=0,$$
which implies that there exist $T_1>0$ and $z\in \mathcal{A}$ such that
\begin{equation}\label{delta_2}
	|\Phi_{T_1}(y_{\delta})-z|<\delta.
\end{equation}
Let $y_1=\Phi_{T_1}(y_{\delta})$ and $T=|y-y_\delta|+T_1+|z-y_1|$. Define $\psi^y\in {\bf AC}_T^y$ as:
\begin{equation*}  \psi^y(t)=
	\begin{cases}
		{\rm LIF}_{y,y_\delta}(t), & t\in \big[0, |y-y_\delta|\big); \vspace{2mm}\\
		\Phi_{t-|y-y_\delta|}(y_{\delta}), & t\in \big[|y-y_\delta|, |y-y_\delta|+T_1\big);\vspace{2mm}\\
		{\rm LIF}_{y_1,z}(t-|y-y_\delta|-T_1), & t\in \big[|y-y_\delta|+T_1, T\big].
	\end{cases}
\end{equation*}
Then, together with \eqref{SLIF_le_L_K}-\eqref{delta_2}, we can estimate
\begin{equation}
\begin{split}
&\mathcal{S}^y_T(\psi^y)= \mathcal{S}^y_{|y-y_\delta|}({\rm LIF}_{y,y_\delta})+0+\mathcal{S}^{y_1}_{|z-y_1|}({\rm LIF}_{\Phi_{T_1}(y_\delta),z})\\
&\quad\ \ \overset{\eqref{SLIF_le_L_K}-\eqref{delta_1}}{\leq} L_1\delta+L_2|z-y_1|
\overset{\eqref{delta_2}}{<} (L_1+L_2)\delta
\leq\eta.
\end{split}\tag*{}
\end{equation}

Thus, we have completed the proof.
\end{proof}

\vspace{1ex}
\phantomsection
\noindent\textbf{Remark 4.2.}\label{rmk-4.2}
For any unordered chain-transitive set $K$, one can similarly define its lower dual attractor and obtain that the quasipotential from $K$ to its lower dual attractor equals zero as well. Let $$H^-(K)=\{x\in \mathbb{R}^r:\ \Phi_s(x)\ll y \text{ for some }y\in K\text{ and } s\ge 0\}.$$
It is known that such $K$ is contained in both the lower boundary of $H^+(K)$ and the upper boundary of $H^-(K)$, which are $1$-codimensional $C^1$-smooth invariant submanifolds, called \textit{nonmonotone manifolds} (see Tere\v{s}\v{c}\'{a}k \cite{Terescak94}, or \cite{H-88,Ta92}).

\vspace{1.5ex}

\begin{prop}[Quasipotential from stationary $p$-arcs]\label{V_y_Ap=0}
Let $J\subset\mathbb{R}^r$ be a stationary $p$-arc with an unstable endpoint.
Then, there exists an attractor $\mathcal{A}=\mathcal{A}(J)$ with $J\cap\mathcal{A}=\emptyset$ such that
\begin{equation}\label{VyAj=0}
	V(y,\mathcal{A})=0\quad \text{ for any }y\in J.
\end{equation}

\end{prop}

\begin{proof}
Without lost of generality, we assume that $p:={\rm sup} J$ is unstable.
To show the existence of the attractor $\mathcal{A}$, we need some preparation.

First, it follows from \cite[Proposition 1.4(i)]{Janusz} that there exist an $\varepsilon > 0$ and a one-dimensional locally forward invariant $C^1$ manifold $W_p^+ \subset B_\varepsilon(p)$ such that $W_p^+$ is simply ordered and $p = \inf W_p^+$. Here, the locally forward invariance of $W_p^+$ means $\Phi_t(x) \in W_p^+$ provided $x \in W_p^+$ and $\Phi_t(x) \in B_\varepsilon(p)$.
Write $B_\varepsilon^+(p)\triangleq B_\varepsilon(p)\cap [p, +\infty]]$. Clearly, $W_p^+\subset B_\varepsilon^+(p)$. Moreover, we can assume without loss of generality that
\begin{equation}\label{E+B+p=empty}
	\mathcal{E}\cap B_{\varepsilon}^+(p)=\{p\}
\end{equation}
(Otherwise, there exists $\{e_n\} \subset \mathcal{E}\backslash\{p\}$ with $e_n \in B_{\frac{1}{n}}^+(p)$. Then, $e_n \to p$ as $n \to \infty$, contradicting to the instability of $p$).

Next, we \textit{claim that the unique orbit corresponding to $W^+_p\backslash\{p\}$ is monotonicity increasing.}
Given any $y\in W^+_p\backslash\{p\}$. By \eqref{E+B+p=empty}, $b(y)\neq0$. Since $W_p^+$ is simply ordered and locally forward invariant, one has either $b(y)<0$ or $b(y)>0$.
In fact, $b(y)>0$ (Otherwise, by \eqref{E+B+p=empty} and the continuity of \( b(y) \) with respect to \( y \), $b(z)<0$ for all $z\in W_p^+\cap[[p,y]]$, contradicting the instability of $p$).
Note also that the Jacobian matrix $\frac{\partial \Phi_t}{\partial x}(x)$ of $\Phi_t$ is strongly positive for all $x\in\mathbb{R}^r$ and $t>0$ (see \cite[Theorem 4.1.1]{Smith}).
Then, we have \( b(\Phi_t(y)) \gg 0 \) for all \( t > 0 \), because the derivative flow preserves the vector field as $b(\Phi_t(y))=\frac{\partial \Phi_t}{\partial x}(y)\cdot b(y)$ for any $t\geq0$.
As a consequence, for any $t_2>t_1>0$, we obtain $\Phi_{t_2}(y)-\Phi_{t_1}(y)=\int_{t_1}^{t_2}b(\Phi_s(y))\text{d}s\gg0,$ which implies that
the orbit of \( y \) is monotonically increasing. Thus, we have proved the claim.

By virtue of the claim, there exists $p_0\in\mathcal{E}$ such that
\begin{equation*}
|\Phi_t(x)-p_0|\to0\ \ \text{ as }\ t\to\infty,\ \quad\text{ for any } x\in W^+_p\backslash\{p\}.
\end{equation*}
Let
\begin{equation*}
	\mathcal{A}=\Lambda\cap\bigcap_{t>0}\Phi_t[p_0,+\infty]].
\end{equation*}
Then the dissipation assumption in (\hyperref[H-1]{H1}) yields that
$\mathcal{A}$ is an attractor satisfying $p\ll\mathcal{A}$ (since $p\ll p_0$), and so $J\cap\mathcal{A}=\emptyset$.
Hence,
\begin{equation*}
\lim_{t\to\infty}{\rm dist}(\Phi_t(x),\mathcal{A})=0\quad \text{ for any } x\in W^+_p\backslash\{p\}.
\end{equation*}

Now, choose $\delta_0>0$ so small that $\overline{J_{\delta_0}}\cap\overline{\mathcal{A}_{\delta_0}}=\emptyset$.
We will prove \eqref{VyAj=0}. For this purpose, fix any $y\in J$.
We will show that, for any $\eta>0$, there exist $T>0$ and $\psi^y\in {\bf AC}_T^y$ satisfying $\psi^y(0)=y$, $\psi^y(T)\in \mathcal{A}$ such that
\begin{equation}\label{attractup_0_j}
	\mathcal{S}^y_{T}(\psi^y)\leq \eta.
\end{equation}
In fact, let $L_1,L_2>0$ be given in Lemma \ref{Vctin} for the compact sets $\overline{B_{\delta_0}(p)}$ and $\overline{\mathcal{A}_{\delta_0}}$, respectively.
Choose the $p$-arc $J_y\ (\subset J)$ connecting $y$ and $p$. Then one can find an increasing smooth function $h$ such that $h([0,1])=J_y$ with $h(0)=y$ and $h(1)=p$. Let $M=\sup|h'(t)|$ and $\lambda(x)$, $x\in J_y$, be the smallest eigenvalue of the positive definite matrix $a(x)$ for system \eqref{itodff}.
Then the compactness of $J_y$ implies that $\inf_{x\in J_y}\lambda(x)=\lambda_0>0$.

For any $\eta>0$, let $\delta=\min\left\{\delta_0,\dfrac{\eta}{2(L_1+L_2)}\right\}>0$ and $\tau=\lambda_0M^{-2}\eta>0$.
Choose a $q\in W^+_p\backslash \{p\}$ and
\begin{equation}\label{j1}
	|q-p|<\delta.
\end{equation}
Noticing
$\lim_{t\rightarrow \infty}{\rm dist}(\Phi_t(q), \mathcal{A})=0$,
one can choose some $T_1>0$ and $z\in\mathcal{A}$ such that
\begin{equation}\label{j2}
	|\Phi_{T_1}(q)-z|<\delta.
\end{equation}
Write $y_1=\Phi_{T_1}(q)$. Now, choose $t_1=\tau^{-1}$ and define numbers $T\geq t_3\geq t_2\geq t_1>0$ as $t_2=t_1+|p-q|$, $t_3=t_2+T_1$ and $T=t_3+|y_1-z|$, by which we define $\psi^y\in {\bf AC}_T^y$ as:
\begin{equation*}  \psi^y(t)=
	\begin{cases}
		h \big(\tau t\big), & t\in \left[0, t_1\right);\vspace{1mm}\\
		{\rm LIF}_{p,q}(t-t_1), & t\in \left[t_1, t_2\right) ;\vspace{1mm}\\
		\vspace{2mm}
		\Phi_{t-t_2}(q), & t\in \left[t_2, t_3\right);\\
		\vspace{1mm}
		{\rm LIF}_{y_1,z}(t-t_3), & t\in \left[t_3, T\right].
	\end{cases}
\end{equation*}
Clearly, $\psi^y(0)=y$ and $\psi^y(T)\in\mathcal{A}$.
Recall that $b(\cdot)$ vanishes on $J_y$.
Then, together with Lemma \ref{Vctin}, Remark \hyperref[rmk-2.1]{2.1} and \eqref{j1}-\eqref{j2}, we have
\begin{align*}
	\mathcal{S}^y_T(\psi^y)= &\ \mathcal{S}^y_{t_1}\big(h(\tau t)\big)+ \mathcal{S}^p_{|q-p|}({\rm LIF}_{p,q}(t-t_1)) + \mathcal{S}^{y_1}_{|z-y_1|}({\rm LIF}_{y_1,z}(t-t_3))\\
	&\overset{{\rm Lemma }\ \ref{Vctin}}{\leq} \ \frac{1}{2}\int_{0}^{t_1}\left(\frac{d}{dt}h(\tau t)\right)^Ta^{-1}\big(h(\tau t)\big)\left(\frac{d}{dt}h(\tau t)\right)dt + L_{1} |p-q| + L_{2} |y_1-z|\\
	&\overset{\eqref{j1}-\eqref{j2}}{\leq} \ \frac{1}{2\lambda_0}\int_{0}^{t_1}\Big|\frac{d}{dt}h(\tau t)\Big|^2dt + (L_1+L_2)\delta\\
	&\quad\quad = \ \frac{\tau^2}{2\lambda_0}\int_{0}^{t_1}\Big|h'(\tau t)\Big|^2dt + (L_1+L_2)\delta	
	\leq\ \frac{\tau^2M^2}{2\lambda_0}\cdot t_1 + (L_1+L_2)\delta.
\end{align*}
Hence, by the definition of $\tau$, $t_1$ and $\delta$, we obtain that
\begin{align*}
	\mathcal{S}^y_T(\psi^y)\leq \frac{\tau^2M^2}{2\lambda_0}\cdot\tau^{-1} + (L_1+L_2)\delta
	\leq \frac{M^2}{2\lambda_0}\tau + \frac{\eta}{2}=\eta.
\end{align*}
Thus, we have obtained \eqref{attractup_0_j}, which completes the proof.
\end{proof}

\section{The Proof of Main Results}\label{proof-of-main-results}

The key point for proving our main result Theorem \ref{Mthm} is guaranteed by Lemma \ref{key-lemma-LDP} below, in which we shall estimate the probability decay rate of rare transition events between the unstable chain-transitive set and its dual attractor.

\begin{lem}\label{key-lemma-LDP}
Assume that {\rm(\hyperref[H-1]{H1})}-{\rm(\hyperref[H-3]{H3})} hold.
Suppose that $K$ is either an unordered chain-transitive set or a stationary $p$-arc with an unstable endpoint.
Then, there exist an attractor $\mathcal{A}$, $\kappa_2>\kappa_1>0$ and $\delta,\eta>0$ such that, for each $y\in K$,
\begin{align}
&{\rm(i)}\quad \liminf_{\varepsilon\to0}\varepsilon^2\log\mathbb{P}\left\{X_{T}^{\varepsilon,z}\in\mathcal{A}_{\eta}\right\}\geq -\kappa_1\quad\text{ uniformly in }z\in B_\delta(y);\label{LDP_int_01}\\
&{\rm(ii)}\quad \limsup_{\varepsilon\to0}\varepsilon^2\log\mathbb{P}\left\{X_{T}^{\varepsilon,z}\notin \mathcal{A}_\eta\right\}\leq -\kappa_2\quad\text{ uniformly in }z\in \mathcal{A}_{\eta},\label{LDP_int_02}
\end{align}
where $T>0$ depends on $y$.
\end{lem}

\begin{proof}
By virtue of Propositions \ref{V-y-A=0} and \ref{V_y_Ap=0}, there exists an attractor $\mathcal{A}$ with $K\cap\mathcal{A}=\emptyset$ such that
\begin{equation}\label{V_B_A=0}
	V(y,\mathcal{A})=0\quad\text{ for any }y\in K.
\end{equation}
Choose $\delta_0>0$ so small that $\overline{K_{\delta_0}}\cap\overline{\mathcal{A}_{\delta_0}}=\emptyset$.
By \cite[Lemma 4.2]{XCJ}, one can find $0<\delta_1<\delta_2<\frac{\delta_0}{2}$ such that
\begin{equation}\label{def_s_1}
	s_1\triangleq V\left(\overline{\mathcal{A}_{\delta_1}},\partial \mathcal{A}_{\delta_2}\right)>0.
\end{equation}
Let $\eta=2\delta_2<\delta_0$.
Then there exists $T_0>0$ such that
\begin{equation}\label{T0}
	\Phi_t(x)\in \mathcal{A}_{\delta_1}\quad \text{for any } x\in\overline{\mathcal{A}_{\eta}} \text{ and }  t\ge T_0.
\end{equation}
Denote
\begin{equation}\label{def_H}
	H=\{\varphi\in {\bf C}_{T_0}:\varphi(0)\in \overline{\mathcal{A}_{\eta}},\ \varphi(T_0)\notin \mathcal{A}_{\delta_1}\}.
\end{equation}
Clearly, $H\subset{\bf C}_{T_0}$ is closed; and moreover, $H$ does not contain any solution of system \eqref{unpersys} (due to (\ref{T0})). Then, Lemma \ref{p-64} entails that
\begin{equation}\label{def_s_2}
	s_2\triangleq\mathcal{S}_{T_0}(H)>0.
\end{equation}
Now, we define
\begin{equation*}
\kappa_1=\frac{1}{2}s_0\quad \text{ and }\quad  \delta=\min\left\{\delta_1,\frac{\kappa_1}{2 L}\right\},
\end{equation*}
where $s_0=\min{\{s_1,s_2\}}>0$, and $L$ is given in Lemma \ref{Vctin} for the compact set $\overline{K_{\delta_0}}$.

For each $y\in K$, we assert that there is a $T>T_0$ such that for each $z\in\overline{B_\delta(y)}$, one can find $\psi^z\in \textbf{AC}_T$ with $\psi^z(0)=z$ and $\psi^z(T)\in \mathcal{A}$ satisfying
\begin{equation}\label{attractup}
	\mathcal{S}^z_{T}(\psi^z)\leq \kappa_1.
\end{equation}
To see this, we recall that Lemma \ref{Vctin} yields that
\begin{equation}\label{attraloc_1}
	\mathcal{S}^z_{|z-y|}({\rm LIF}_{z,y})\leq L|z-y|\leq L\delta\leq \frac{\kappa_1}{2}.
\end{equation}
Furthermore, by \eqref{V_B_A=0}, we can find $T_1>0$ and $\varphi\in\textbf{AC}_{T_1}$ with $\varphi(0)=y$ and $\varphi(T_1)\in \mathcal{A}$ such that
\begin{equation}\label{attraloc_2}
	\mathcal{S}^y_{T_1}(\varphi)\leq \frac{\kappa_1}{2}.
\end{equation}
Now, by letting $T=\delta+T_1+T_0>0$, we define $\psi^z\in \textbf{AC}_T$ as
\begin{equation*}\label{psi^z}
	\psi^z(t)=\left\{
	\begin{array}{ll}
		{\rm LIF}_{z,y}(t), & t\in [0, |y-z|),\vspace{2mm} \\
		\varphi(t-|y-z|), & t\in [|y-z|, |y-z|+T_1), \vspace{2mm} \\
		\Phi_{t-|y-z|-T_1}(\varphi(T_1)), & t\in[|y-z|+T_1, T].
	\end{array}
	\right.	
\end{equation*}
Clearly, $T>T_0$, $\psi^z(0)=z$ and $\psi^z(T)\in\mathcal{A}$ (due to the invariance of $\mathcal{A}$); and moreover, \eqref{attraloc_1}-\eqref{attraloc_2} implies \eqref{attractup}.
Thus, we have proved the assertion.

Now, we are ready to prove \eqref{LDP_int_01}-\eqref{LDP_int_02}.
We first prove \eqref{LDP_int_01}. To this  purpose, for any $z\in\overline{B_\delta(y)}$, we choose $\psi^z$ in the assertion satisfying \eqref{attractup}.
Then, $\psi^z(T)\in\mathcal{A}$ and
\begin{equation}\label{G^z_include}
	\left\{\rho_T(X^{\varepsilon,z}_\cdot,\psi^z)<\eta \right\}\subset \left\{X^{\varepsilon,z}_T\in\mathcal{A}_{\eta}\right\}.
\end{equation}
Take $\delta$, $\gamma$, $s_0$ and $K$ as $\eta$, $\gamma$, $\kappa_1$ and $\overline{B_\delta(y)}$ in (\hyperref[Lower]{L}), respectively.
Given any $z\in \overline{B_\delta(y)}$, it follows from \eqref{ufldplbb}, \eqref{attractup} and \eqref{G^z_include} that there exists $\varepsilon_1>0$ such that
\begin{equation*}\label{LDP_int_1_1}
	\begin{split}
		\mathbb{P}\left\{X_T^{\varepsilon,z}\in\mathcal{A}_{\eta}\right\} \overset{\eqref{G^z_include}}{\geq}&\mathbb{P}\left\{\rho_T(X^{\varepsilon,z}_\cdot,\psi^z)<\eta\right\}\\
		\overset{\eqref{ufldplbb}}{\geq}&\exp\left\{-\frac{\mathcal{S}^z_T(\psi^z)+\gamma}{\varepsilon^2}\right\}
		\overset{\eqref{attractup}}{\geq} \exp\left\{-\frac{\kappa_1+\gamma}{\varepsilon^2}\right\},
	\end{split}
\end{equation*}
or equivalently,
\begin{equation*}
	\varepsilon^2\log \mathbb{P}\left\{X_T^{\varepsilon,z}\in\mathcal{A}_{\eta}\right\}\geq -\kappa_1-\gamma,
\end{equation*}
for any $0<\varepsilon<\varepsilon_1$.
Thus, we have proved \eqref{LDP_int_01}.

Finally, we prove \eqref{LDP_int_02}.
Given any $z\in \overline{\mathcal{A}_{\eta}}$, let
\begin{equation}\label{def_F_z}
	F^z = \{\varphi\in \textbf{AC}_T^z: \varphi(T)\notin  \mathcal{A}_{\eta}\}
\end{equation}
and
\begin{equation*}\label{def_G_z}
	G^z=\{\varphi^z\in\mathbf{AC}_T^z:\rho_T(\varphi^z,F^z)<\delta_2\}.
\end{equation*}
We \textit{claim that}
\begin{equation*}\label{inf_S_A_delta1}
	\inf_{z\in\overline{\mathcal{A}_{\eta}}}\inf_{\varphi^z\in G^z}\mathcal{S}^z_T(\varphi^z)\geq s_0.
\end{equation*}
In fact, for any $z\in \overline{\mathcal{A}_{\eta}}$ and $\varphi^z\in G^{z}$, choose a $\varphi^z_1\in F^{z}$ such that $\rho_T(\varphi^z_1,\varphi^z)<\delta_2$,
which implies that
\begin{equation}\label{phi_notin_Adelta1*}
	\varphi^z(T)\notin\mathcal{A}_{\delta_2},
\end{equation}
since $\eta=2\delta_2$ and $\varphi^z_1(T)\notin \mathcal{A}_{\eta}$.
Without loss of generality, we assume that $\mathcal{S}^z_T(\varphi^z)<s_2$ (Otherwise, $\mathcal{S}^z_T(\varphi^z)\geq s_2(\geq s_0)$, which yields the claim directly).
Since $T>T_0$, one has $\mathcal{S}^z_{T_0}(\varphi^z)\leq \mathcal{S}^z_T(\varphi^z)<s_2$.
Together with \eqref{def_H}-\eqref{def_s_2}, we obtain
\begin{equation*}\label{phi_in_Adelta2}
	\varphi^z(T_0)\in\mathcal{A}_{\delta_1}.
\end{equation*}
Hence, by noticing that $\overline{\mathcal{A}_{\delta_1}}\subset\mathcal{A}_{\delta_2}$ and \eqref{phi_notin_Adelta1*}, it entails that there exists $t_1\in(T_0,T]$ such that $\varphi^z(t_1)\in\partial\mathcal{A}_{\delta_2}$.
Then the definition of quasipotential $V$ and \eqref{def_s_1} implies that $$\mathcal{S}^z_T(\varphi^z)\geq V(\overline{\mathcal{A}_{\delta_1}},\partial\mathcal{A}_{\delta_2})=s_1(\geq s_0),$$
which completes the claim.

Now, let $\kappa_2=\frac{3}{4}s_0>0$. Then the claim yields that
\begin{equation*}
	\mathbb{F}_T^z(\kappa_2)\cap G^z=\emptyset\quad\text{ for any } z\in \overline{\mathcal{A}_{\eta}}.
\end{equation*}
In other words,
\begin{equation*}
	\mathbb{F}_T^z(\kappa_2)\cap\left\{\varphi\in\mathbf{C}_T^z:\rho_T(\varphi,F^z)<\delta_2\right\}=\emptyset,
\end{equation*}
which implies that
\begin{equation}\label{F_belong_F0.9}
	F^z\subset\{\varphi\in\textbf{C}_T^z: \rho_{T}(\varphi,\mathbb{F}_{T}^{z}(\kappa_2))\geq\delta_2\}.
\end{equation}
Take $\delta$, $\gamma$, $s_0$ and $K$ as $\delta_2$, $\gamma$, $\kappa_2$ and $\overline{\mathcal{A}_{\eta}}$ in (\hyperref[Upper]{U}), respectively.
Given any $z\in\overline{\mathcal{A}_{\eta}}$, together with \eqref{def_F_z} and \eqref{F_belong_F0.9}, one can find an $\varepsilon_2>0$ such that
\begin{align*}
	\mathbb{P}\left\{X_T^{\varepsilon,z}\notin\mathcal{A}_{\eta}\right\} \overset{\eqref{def_F_z}}{=}&\mathbb{P}\left\{X^{\varepsilon,z}_\cdot\in F^z\right\}\\
	\overset{\eqref{F_belong_F0.9}}{\leq}&\mathbb{P}\left\{\rho_{T}(X^{\varepsilon,z}_\cdot,\mathbb{F}_{T}^{z}(\kappa_2))\geq\delta_2\right\}
	\overset{\eqref{ufldpupbb}}{\leq}\exp\left\{-\frac{\kappa_2-\gamma}{\varepsilon^2}\right\};
\end{align*}
and hence,
\begin{equation*}
	\varepsilon^2\log\mathbb{P}\left\{X_T^{\varepsilon,z}\notin\mathcal{A}_{\eta}\right\}\leq -\kappa_2+\gamma,
\end{equation*}
for any $0<\varepsilon<\varepsilon_2$. Thus, we have proved \eqref{LDP_int_02}, which completes the proof.
\end{proof}

\vspace{2ex}

Now, we are ready to prove our main results.

\begin{proof}[Proof of Theorem \ref{Mthm}]
Let $\mu$ be a zero-noise limit of the system \eqref{itodff} (the assumption (\hyperref[H-2]{H2}) ensures the existence of $\mu$), and $K$ be any connected component of ${\rm supp}(\mu)$.
Then $K$ is a chain-transitive set. By virtue of Lemma \ref{Hirsch_99},
one of the following three alternatives must occur:\par
(i) $K$ is unordered; or otherwise,\par
(ii) $K$ is a stationary $p$-arc with one unstable endpoint;\par
(iii) $K$ is a stationary $p$-arc of stable equilibria.\\
We will prove, in any case of (i) or (ii), $\mu(K)=0$.
To this purpose, fix any $y\in K$, Lemma \ref{key-lemma-LDP} implies that for case (i) and (ii),
there exists an attractor $\mathcal{A}$,
and numbers $\kappa_2>\kappa_1>0$ and $T,\delta,\eta>0$ such that \eqref{LDP_int_01}-\eqref{LDP_int_02} holds.

Denote
\begin{equation*}
I_1=\int_{B_{\delta}(y)}\mathbb{P}\left\{X_{T}^{\varepsilon,z}\in\mathcal{A}_{\eta}\right\}\mu^{\varepsilon}(dz)\quad \text{ and }\quad I_2=\int_{\mathcal{A}_{\eta}}\mathbb{P}\left\{X_T^{\varepsilon,z}\notin\mathcal{A}_{\eta}\right\}\mu^\varepsilon(dz),
\end{equation*}
where $\mu^\varepsilon$ is a stationary measure of \eqref{itodff}.
Consequently, by choosing $\gamma\in(0,\frac{\kappa_2-\kappa_1}{2})$, we have
\begin{equation}\label{LDP_int_1}
	\mu^{\varepsilon}(B_{\delta}(y))\leq I_1\exp\left\{\frac{\kappa_1+\gamma}{\varepsilon^2}\right\}\quad\text{ and }\quad \mu^\varepsilon(\mathcal{A}_{\eta})\geq I_2\exp\left\{\frac{\kappa_2-\gamma}{\varepsilon^2}\right\},
\end{equation}
for any small $\varepsilon>0$.
This implies that
\begin{equation}\label{LDP_int_3}
\begin{split}
I_1&\leq\int_{(\mathcal{A}_\eta)^c}\mathbb{P}\left\{X_{T}^{\varepsilon,z}\in\mathcal{A}_{\eta}\right\}\mu^{\varepsilon}(dz)\\
&= \left(\int_{\mathbb{R}^r}-\int_{\mathcal{A}_{\eta}}\right)\mathbb{P}\left\{X_{T}^{\varepsilon,z}\in\mathcal{A}_{\eta}\right\}\mu^{\varepsilon}(dz)\\
&=\int_{\mathcal{A}_{\eta}}\mathbb{P}\left\{X_{T}^{\varepsilon,z}\notin\mathcal{A}_{\eta}\right\}\mu^{\varepsilon}(dz)
=I_2,
\end{split}
\end{equation}
which means that $I_1\leq I_2$. It then follows from \eqref{LDP_int_1}-\eqref{LDP_int_3} that
\begin{align*}
\frac{\mu^\varepsilon(B_\delta(y))}{\mu^\varepsilon(\mathcal{A}_{\eta})}
\overset{\eqref{LDP_int_1}}{\leq} \frac{\exp\left\{\frac{\kappa_1+\gamma}{\varepsilon^2}\right\} I_1}{\exp\left\{\frac{\kappa_2-\gamma}{\varepsilon^2}\right\}I_2}
\overset{\eqref{LDP_int_3}}{\leq} \exp\left\{-\frac{(\kappa_2-\kappa_1)-2\gamma}{\varepsilon^2}\right\}.
\end{align*}
Since $\mu^\varepsilon(\mathcal{A}_\eta)\leq1$, we obtain
$$\mu^\varepsilon(B_\delta(y))\leq \exp\left\{-\frac{(\kappa_2-\kappa_1)-2\gamma}{\varepsilon^2}\right\}.$$
Recall that $\gamma<\frac{1}{2}(\kappa_2-\kappa_1)$, it entails that $$\mu(B_\delta(y))\leq\liminf_{\varepsilon\to0}\mu^\varepsilon\big(B_{\delta}(y)\big)=0.$$
Noticing $K\subset \bigcup_{y\in K}B_{\delta}(y)$ ($\delta$ may depend on $y$), it then follows from the compactness of $K$ and the subadditivity of $\mu$ that there exists a finite subcover $\{B_{\delta_i}(y_i)\}_{i=1}^{N}$ of $K$ such that $$\mu(K)\leq \sum_{i=1}^{N}\mu\big(B_{\delta_i}(y_i)\big)=0.$$
Thus, only the alternative (iii) holds. We have completed the proof.
\end{proof}

\section{Applications to Stochastic Biochemical Control Circuit}\label{app}

In this section, we will apply our main results to investigate the stochastic stability of a well-known biochemical control circuit model
\begin{eqnarray}\label{dbcc}
\left\{\begin{array}{ll}
	\dot{x_{1}}(t)=f(x_{r}(t))-\alpha_{1}x_{1}(t),\\
	\dot{x_{j}}(t)=x_{j-1}(t)-\alpha_{j}x_{j}(t),\quad 2\leq j\leq r,\\
\end{array}
\right.
\end{eqnarray}
where the parameters $\alpha_j (1\leq j\leq r)$ are positive constants,
and the function $f:\mathbb{R}\to\mathbb{R}$ is a smooth function satisfying $f'\geq0$.
Following \cite{Selgrade2,Selgrade,Griffith}, we refer to system \eqref{dbcc} as a \textit{positive feedback system}. For such systems with $r = 2$ or $r = 3$, Selgrade \cite{Selgrade2,Selgrade} demonstrated that every positive-time trajectory of system \eqref{dbcc} converges to an equilibrium under mild additional restrictions. Jiang \cite{Jiang3, Jiang4} introduced a new criterion for the convergence of solutions of \eqref{dbcc} with $r = 3$ and $r = 4$. However, in $5$ dimension, system \eqref{dbcc} can exhibit more complicated behaviors, such as Hopf bifurcation (see more details in \cite{Selgrade1}).

Now, we consider the perturbation of the deterministic system \eqref{dbcc} as
\begin{eqnarray}\label{sbcc}
	\left\{\begin{array}{ll}
		dx_{1}=(f(x_{r})-\alpha_{1}x_{1})dt+\varepsilon \sigma_1(x_1)dW_t^1,\\
		dx_{j}=(x_{j-1}-\alpha_{j}x_{j})dt +\varepsilon \sigma_j(x_j)dW_t^j,\quad 2\leq j\leq r,
	\end{array}
	\right.
\end{eqnarray}
where the perturbation parameter $\varepsilon>0$ is small, and $W_t=(W^{1}_t,\cdots,W^{r}_t)^T$ is a standard $r$-dimensional Wiener process, and the diffusion function $\sigma_j(x_j)$ is locally Lipschitz continuous with respect to $x_j$ for $j=1,\cdots,r$.

\vspace{1.5ex}
In the following, we focus on the Griffith-Type, where $f$ satisfies
\begin{flalign}\label{G-m}
&{\rm\bf(G)}\quad\quad\quad\ \ \ f(z)={\rm sgn}(z)\frac{|z|^m}{1+|z|^m}\quad \text{ for any } z\in\mathbb{R},\text{ with the parameter }m\geq1.\tag*{}&
\end{flalign}

\begin{prop}\label{G}
Assume that system \eqref{dbcc} satisfies {\rm(\hyperref[G-m]{G})}. Assume also that there exists $c>0$ such that
\begin{equation}\label{sigma-leq-c}
	0<\sigma_j^2(z)\le c(z^2+1)\quad\text{ for any }z\in \mathbb{R}\text{ and }1\leq j\leq r.
\end{equation}
Let $\mu$ be a zero-noise limit of \eqref{sbcc}. Then
\begin{equation*}
\mu=\sum_{i=1}^n\lambda_i\delta_{E_i}(\cdot)\quad\text{ with }\sum_{i=1}^n\lambda_i=1,
\end{equation*}
where $E_i$, $1\leq i\leq n$, are asymptotically stable equilibria of {\rm(\ref{dbcc})}.
\end{prop}

\begin{proof}
Due to the form of system \eqref{dbcc}, we write the matrix
\begin{equation*}
A=
\begin{pmatrix}
-\alpha_1 & & & 0 \\
1 & -\alpha_2 & \\
 & \ddots & \ddots \\
 & & 1 & -\alpha_r
\end{pmatrix}
\end{equation*}
for brevity.
Then all the eigenvalues of $A$ are negative. By the well-known Lyapunov theorem (e.g. see \cite[Section 6.4.2 on p.249]{Ortega}), there is an $(r\times r)$ positive definite matrix $B$ such that
$$A^TB+BA=-I_r,$$
where $I_r$ is the $(r\times r)$ identity matrix.
Define the positive definite Lyapunov function $$V(x)=x^TBx.$$
Clearly, one has $\nabla V(x)=2Bx$ and $D^2 V(x)=2B$.
Then, together with $|f(x_r)|=\frac{|x_r|^m}{1+|x_r|^m}$, we have
\begin{align}\label{b-V-1}
	\langle b(x),\nabla V(x)\rangle
	=-|x|^2 +f(x_r)\frac{\partial V}{\partial x_1}
	\leq -|x|^2+2\sum_{i=1}^r|b_{1i}x_i|.
\end{align}
By letting $R>0$ sufficiently large, we obtain that system \eqref{dbcc} is dissipative.

To verify the assumption (\hyperref[H-1]{H1}), we note that the Jacobian matrix of system \eqref{dbcc}
\begin{equation*}
	J(x) =
	\begin{pmatrix}
		-\alpha_1 & & & f'(x_r) \\
		1 & -\alpha_2 & \\
		& \ddots & \ddots \\
		& & 1 & -\alpha_r
	\end{pmatrix}.
\end{equation*}
Since $f'\geq0$ in (\hyperref[G-m]{G}), it is clear that system \eqref{dbcc} is cooperative.
As for the irreducibility, we consider the case (i) $m=1$ and case (ii) $m>1$, respectively.

Case (i): $m=1$. In this case, one has \( f'(z) = \frac{1}{(1+|z|)^2} > 0 \), which implies that $J(x)$ is irreducible for any $x\in\mathbb{R}^r$. Therefore, the solution flow of \eqref{dbcc} is strongly monotone (see \cite[Theorem 4.4.1]{Smith}).

Case (ii): $m>1$. In this case, \( f'(z) = \frac{m|z|^{m-1}}{(1+|z|^m)^2} \geq 0 \) holds for all \( z \in \mathbb{R} \).
Hence, \textit{$J(x)$ is only irreducible outside \(\{ x \in \mathbb{R}^r : x_r = 0 \}\)}.
In such situation, we will still obtain that
\begin{equation}\label{Jacobi-Phi-gg0}
	\frac{\partial \Phi_t}{\partial x}(x) \gg 0,\quad\text{ for any }x\neq0\text{ and }t>0.
\end{equation}
For this purpose, by repeating the same arguments in Smith \cite[Theorem 4.1.1]{Smith}, one only need to show that \textit{for any $x\neq0$ and any interval $(t_1,t_2)$, there exists $t_0\in(t_1,t_2)$ such that $J(\Phi_{t_0}(x))$ is irreducible} (see \cite[p.57, the paragraph below Theorem 4.1.1]{Smith}).
Suppose not, then
there exist \( x \neq 0 \) and \( 0 < t_1 < t_2 \) such that \( J(\Phi_t(x)) \) is reducible for any \( t \in (t_1, t_2) \). Equivalently, this means \( f'((\Phi_t(x))_r) = 0 \) (hence $(\Phi_t(x))_r = 0$), for all $t \in (t_1, t_2)$. Here, the subscript \( r \) denotes the \( r \)-th component. So, the time-derivative \( \frac{d}{dt}(\Phi_t(x))_r = 0 \) for all \( t \in (t_1, t_2) \). By the $r$-th equation of system \eqref{dbcc}, we have
\begin{equation*}
	(\Phi_t(x))_{r-1} = 0, \quad \text{for all } t \in (t_1, t_2).
\end{equation*}
By iterating the same arguments, we obtain from the form of system \eqref{dbcc} that \( (\Phi_t(x))_i = 0 \), for all \( t \in (t_1, t_2) \) and $i=1,2,\cdots,r-1$, which implies \( \Phi_t(x) = 0 \) for \( t \in (t_1, t_2) \). Since \( \{0\} \) is the unique equilibrium for system \eqref{dbcc}, it yields that $x=0$, a contradiction. Thus, we have proved \eqref{Jacobi-Phi-gg0}.

Consequently, \eqref{Jacobi-Phi-gg0} directly implies
\begin{equation*}
	\Phi_t(y) - \Phi_t(x) = \int_{0}^{1} \frac{\partial \Phi_t}{\partial x}\big(x + r(y - x)\big) \cdot (y - x) \, dr \gg 0,
\end{equation*}
whenever $x<y$ and $t>0$. Therefore, the solution flow \( \Phi_t \) is strongly monotone.

Next, we will show that $V$ satisfies both (\hyperref[H-2]{H2}) and (\hyperref[H-3]{H3}).

For (\hyperref[H-2]{H2}), it is clear $V$ satisfies \eqref{Vinfty}. Moreover, \eqref{sigma-leq-c}-\eqref{b-V-1} implies that
\begin{align*}
	\langle b(x),\nabla V(x)\rangle +\frac{\varepsilon^2}{2}{\rm Tr}\big(\sigma^T(x)D^2V(x)\sigma(x)\big)
	\leq  -|x|^2+2\sum_{i=1}^r|b_{1i}x_i|+c\varepsilon^2{\rm Tr}(B)\cdot(|x|^2+1).
\end{align*}
Then, by choosing $\varepsilon_0>0$ small and $R>0$ sufficiently large, if necessary, one obtains \eqref{Vdissip} with $\gamma>0$. This confirms (\hyperref[H-2]{H2}).

As for (\hyperref[H-3]{H3}), $V$ satisfies \eqref{21} clearly.
To verify \eqref{22}, we only need to estimate $\frac{|\sigma^T(x) \cdot \nabla V (x)|^2}{\eta V(x)}$, where $\eta$ will by determined later. To this purpose, let $\lambda_1>0$ be the least eigenvalue of $B$. Then
$$\frac{|\sigma^T(x) \cdot \nabla V (x)|^2}{\eta V(x)}\le \frac{ \|\sigma(x)\|^2\cdot|\nabla V(x)|^2}{\eta V(x)}\overset{\eqref{sigma-leq-c}}{\leq} \frac{2c(|x|^2+r)|Bx|^2}{\eta \lambda_1|x|^2}\le \frac{2c\|B\|^2(|x|^2+r)}{\eta \lambda_1},$$
by which we obtain
\begin{equation*}
	\begin{split}
		({\rm LHS}) &\leq -|x|^2 + 2\sum_{i=1}^r |b_{1i}x_i| + c\theta\mathrm{Tr}(B)(|x|^2 + 1) + \frac{2c\|B\|^2(|x|^2 + r)}{\eta \lambda_1} \\
		&= \underbrace{\left( -1 + c\theta\mathrm{Tr}(B) + \frac{2c\|B\|^2}{\eta\lambda_1} \right)}_{\Gamma_1}|x|^2 + 2\sum_{i=1}^r |b_{1i}x_i| + \left( c\theta\mathrm{Tr}(B) + \frac{2cr\|B\|^2}{\eta\lambda_1} \right),
	\end{split}
\end{equation*}
where {\rm (LHS)} means the left hand side of \eqref{22}. By selecting $\theta > 0$ sufficiently small and $\eta > 0$ sufficiently large to ensure $\Gamma_1 < 0$, and subsequently choosing $R > 0$ large enough such that $({\rm LHS}) < 0$ for $|x| > R$, we obtain
$$({\rm LHS}) \leq M \leq M(1 + V(x)) \quad \text{ for any } x \in \mathbb{R}^r,$$
where $M := \max_{|x| \leq R} ({\rm LHS})$. This deduces \eqref{22}.

Finally, since
$${\rm Tr}\big(\sigma^T(x)D^2V(x)\sigma(x)\big)\geq 2\sum_{i=1}^rb_{ii}\sigma_i^2(x)\ge 0,$$
which implies \eqref{23}.
Thus, $V$ satisfies (\hyperref[H-3]{H3}).

Therefore, the system \eqref{dbcc} satisfies (\hyperref[H-1]{H1})-(\hyperref[H-3]{H3}).
Clearly, there exists a finite number of equilibria in $\mathbb{R}^r$.
Then Theorem \ref{Mthm} (or Corollary \ref{Mthm-c}) directly implies Proposition \ref{G}.
\end{proof}

According to Proposition \ref{G}, in order to locate the support of
the zero-noise limits of \eqref{sbcc}, one needs to find out all the asymptotically stable equilibria of \eqref{dbcc}.
Denote by $\mathcal{E}_s$ the set of all asymptotically stable equilibria.
For Griffith-Type, $O$ is an equilibrium automatically.
To look for other non-zero equilibria, it suffices to look for the positive solution of
\begin{equation}\label{fun-h}
h(z)\triangleq\frac{f(z)}{z}=\phi,
\end{equation}
where $\phi=\prod_{i=1}^r\alpha_i$.
Let $\mathcal{V}_0=\big(\Pi_{i=2}^r\alpha_i,\Pi_{i=3}^r\alpha_i,\cdots,\alpha_r,1\big)^T\in\mathbb{R}^r$. Then any equilibrium possesses the form $z\mathcal{V}_0$ with a scalar $z$.
Recall that the Griffith-Type function $f(z)={\rm sgn}(z)\frac{|z|^m}{1+|z|^m}$ is odd.
Then, the set of equilibria
\begin{equation*}
\mathcal{E}=\{O\}\cup \left\{\pm z\mathcal{V}_0 \Big| \frac{z^{m-1}}{1+z^m}=\phi \text{ for }z> 0\right\},
\end{equation*}
where $``\pm"$ means plus-minus.
Moreover, we have
\begin{lem}\label{g-h-prime}
Assume {\rm(\hyperref[G-m]{G})} holds. Then, the equilibrium $z\mathcal{V}_0$ with $z>0$ is asymptotically stable if $h'(z) < 0$, and unstable if $h'(z) > 0$, where $h$ is defined in \eqref{fun-h}. Similarly, the equilibrium $z\mathcal{V}_0$ satisfying $z<0$ is asymptotically stable if $h'(z) > 0$, and unstable if $h'(z) < 0$.
\end{lem}
\begin{proof}
For the equilibrium $z\mathcal{V}_0$,
we hereafter only focus on the case that $z>0$.
One has $h(z)=\phi$; and moreover, $$h'(z)=\frac{zf'(z)-f(z)}{z{^2}}=\frac{f'(z)-\phi}{z}.$$
Now, define $g(z)\triangleq f(z)-\phi z$. We obtain $g'(z)=f'(z)-\phi$, hence, $$g'(z)=z h'(z).$$
If $h'(z)<0$, then $g'(z)<0$. It follows from \cite[Corollary 5.5]{Selgrade} that $z\mathcal{V}_0$ is asymptotically stable.
If $h'(z)>0$, we have $g'(z)>0$. Choose $\varepsilon>0$ so small that $g'(z_1)>0$ for any $z_1\in (z-\varepsilon,z+\varepsilon)$. This implies $g(z_1)<0$ for any $z_1\in (z-\varepsilon,z)$ and $g(z_1)>0$ for any $z_1\in(z,z+\varepsilon)$ (since $g(z)=0$).
By \cite[Theorem 6.1]{Selgrade}, $z\mathcal{V}_0$ is unstable.

Thus, we conclude the proof.
\end{proof}

Case (i): For $m=1$. Then $h(z)=\dfrac{1}{1+z}$ for $z>0$, whose graph is illustrated in \ref{Fig.4}.

Clearly, if $\phi\geq1$, then $O$ is the unique equilibrium which is stable. While, if
$ 0<\phi<1$, then $O$ is a saddle, $\pm h^{-1}(\phi)\mathcal{V}_0$ are asymptotically stable, since $h'\left(h^{-1}(\phi)\right)<0$ and $h'\left(-h^{-1}(\phi)\right)>0$ (see Lemma \ref{g-h-prime}).
Thus,
\begin{equation*}
\mathcal{E}_s=\left\{
\begin{array}{ll}
	\{O\}, & {\rm if}\ \phi\ge 1; \\
	\{\pm h^{-1}(\phi)\mathcal{V}_0\}, & {\rm if}\ 0<\phi<1;
\end{array}
\right.
\end{equation*}
and hence, any zero-noise limit $\mu$ in Proposition \ref{G} is
\begin{equation*}
\mu=\left\{
\begin{array}{ll}
	\delta_O(\cdot), & {\rm if}\ \phi\ge 1; \\
	\lambda_1\delta_{h^{-1}(\phi)\mathcal{V}_0}(\cdot)+\lambda_2\delta_{-h^{-1}(\phi)\mathcal{V}_0}(\cdot), & {\rm if}\ 0<\phi<1,
\end{array}
\right.
\end{equation*}
where $\lambda_1+\lambda_2=1$, with $\lambda_1,\lambda_2\geq0$.

\begin{figure}[H]
	\centering
	\subfigure[($m=1$).]{
		\label{Fig.4}
		\includegraphics[scale=0.6]{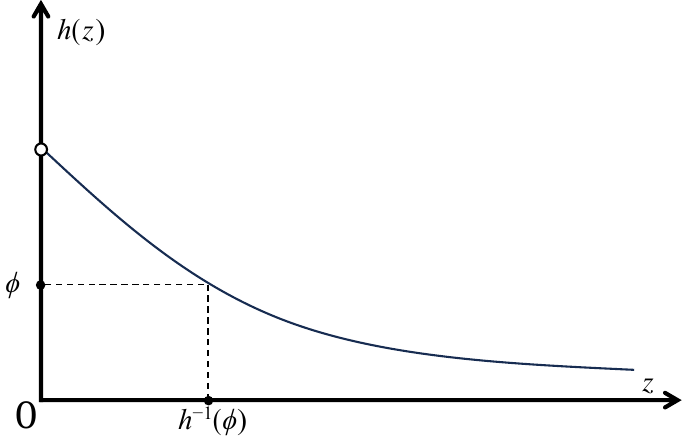}}
	\subfigure[($m>1$): $z_m$ is the unique maximal value point with $h(z_m)=\phi_m$; and $z_i$, $i=1,2$, are the pre-image of $\phi$.]{
		\label{Fig.5}
		\includegraphics[scale=0.6]{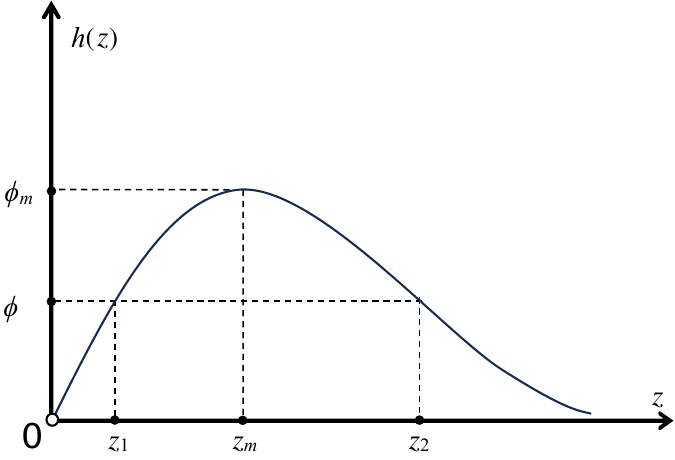}}
	\caption{The graph of $h(z)$ for $z\geq0$.}
	\label{Fig.main1}
\end{figure}

Case (ii): For $m>1$. $h(z)=\dfrac{z^{m-1}}{1+z^m}$ for $z\geq0$, whose graph is illustrated in \ref{Fig.5}.
It is not difficult to see that the corresponding set of equilibria $\mathcal{E}$ is as
\begin{equation*}
\mathcal{E}=\left\{
\begin{array}{ll}
	\{O\}, & {\rm if}\ \phi> \phi_m; \\
	\{O,\ \pm z_m\mathcal{V}_0 \}, & {\rm if}\ \phi=\phi_m;\\
	\{O,\ \pm z_1\mathcal{V}_0, \ \pm z_2\mathcal{V}_0 \}, & {\rm if}\ 0< \phi<\phi_m.
	
\end{array}
\right.
\end{equation*}

If $\phi>\phi_m$, then $\{O\}$ is the unique equilibrium which is global asymptotically stable, that is, a global attractor.
If $\phi=\phi_m$, then from \ref{Fig.5} it follows that $g(z)\triangleq f(z)-\phi_mz=z(h(z)-\phi_m)<0$, for any $z\neq z_m$. Therefore, by \cite[Lemma 5.1]{Selgrade}, $\Phi_t(z\mathcal{V}_0)$ decreasingly converges to $O$ as $t\rightarrow \infty$ whenever $z\in(0,z_m)$, and $\Phi_t(z\mathcal{V}_0)$ decreasingly converges to $z_m\mathcal{V}_0$ as $t\rightarrow \infty$ whenever $z\in (z_m,+\infty)$. Thus, the equilibrium $z_m\mathcal{V}_0$ is unstable. Similarly, we obtain that $-z_m\mathcal{V}_0$ is also unstable.

If $0< \phi<\phi_m$, then $h'(z_1)>0$ and $h'(z_2)<0$, which means that $\pm z_1\mathcal{V}_0$ are both unstable; and hence, $O$ is asymptotically stable (see \cite[Lemma 5.1]{Selgrade}). And also, $\pm z_2\mathcal{V}_0$ are both asymptotically stable. Consequently, we have
\begin{equation*}
	\mathcal{E}_s=\left\{
	\begin{array}{ll}
		\{O\}, & {\rm if}\ \phi\geq \phi_m; \\
		\{O,\ \pm z_2\mathcal{V}_0\}, & {\rm if}\ 0< \phi<\phi_m.
	\end{array}
	\right.
\end{equation*}
Together with Proposition \ref{G}, we obtain that the zero-noise limit
\begin{equation}
	\mu=\left\{
	\begin{array}{ll}
		\delta_O(\cdot), & {\rm if}\ \phi\ge \phi_m, \\
		\lambda_1\delta_O(\cdot)+\lambda_2\delta_{z_2\mathcal{V}_0}(\cdot)+\lambda_3\delta_{-z_2\mathcal{V}_0}(\cdot), & {\rm if}\ 0<\phi<\phi_m,\\
	\end{array}
	\right.
\end{equation}
where $\sum_{i=1}^{3}\lambda_i=1$ with $\lambda_i\geq0$, for $i=1,2,3$.

\vspace{1ex}
Due to our analysis, we present the classification of all zero-noise limits for stochastic Griffith-Type model in the Table \ref{table2}.

\begin{table}[h]
\centering
\caption{The classification of zero-noise limits for Griffith-Type Model (with $\phi=\Pi_{i=1}^r\alpha_i$).\label{table2}}
\begin{tabular}{|c|c|c|c|c|}
	\hline
	\multirow{2}{*}{Parameters}  & \multicolumn{2}{c|}{\multirow{1}{*}{$m=1$}} &
	\multicolumn{2}{c|}{\multirow{1}{*}{$m>1$}}\\
	\cline{2-5}
	\multirow{2}{*}{}& $\phi\ge1$ &  $0<\phi<1$ & $\phi\ge\phi_m$ & $0<\phi<\phi_m$ \\
	
	\hline
	\tabincell{c}{Limiting\\
		measures} & $\delta_{O}$ & $\lambda_1\delta_{h^{-1}(\phi)\mathcal{V}_0}+\lambda_2\delta_{-h^{-1}(\phi)\mathcal{V}_0}$ & $\delta_{O}$ & $\lambda_1\delta_{O}+\lambda_2\delta_{z_2\mathcal{V}_0}+\lambda_3\delta_{-z_2\mathcal{V}_0}$\\
	\hline
\end{tabular}
\end{table}

\vspace{4mm}

\noindent\textbf{Remark 6.1.} It deserves to point out that, for $m>1$ and $0<\phi<\phi_m$, system \eqref{dbcc} may admit a nontrivial periodic orbit (see Selgrade \cite{Selgrade1} for the 5-dimensional case) via the Hopf bifurcation. More precisely, let $r=5$ and $\alpha_i=\alpha$ for any $i=1,\cdots,5$. Then system \eqref{dbcc} admits a unique Hopf bifurcating point given by
$$\beta=\left(\frac{\eta^{m-1}}{1+\eta^m}\right)^{\frac{1}{5}}$$
where
\begin{equation*}
\eta=\left(m\cos^5\left(\frac{2\pi}{5}\right)-1\right)^{\frac{1}{m}}\quad\text{ with }\ m>\frac{1}{\cos^5\big(\frac{2\pi}{5}\big)}\thickapprox 305.
\end{equation*}
In fact, Selgrade \cite{Selgrade1} proved that there is a small $\gamma>0$ such that, for any $\alpha\in(\beta-\gamma,\beta)$, system \eqref{dbcc} admits two periodic orbits $\gamma^\pm$ lying on 4-dimensional Lipschitz nonmonotone manifolds $\mathcal{H}^\pm\subset \pm{\rm Int}\mathbb{R}^5_+$, respectively, where ${\rm Int}\mathbb{R}^5_+=\{x\in\mathbb{R}^5:x_i>0 \text{ for each }i=1,\cdots,5\}$.
Furthermore, $\pm z_1\mathcal{V}_0\in \mathcal{H}^\pm$. Thus, our Proposition \ref{G} excludes the concentration of zero-noise limit on periodic orbits.

\end{document}